\documentclass[12pt]{article}
\usepackage{amsmath,amssymb,amsthm,graphicx,color,bbm}
\usepackage[left=1in,right=1in,top=1in,bottom=1in]{geometry}
\usepackage[british]{babel}
\usepackage{hyperref} 
\usepackage{enumitem}
\usepackage{csquotes,caption}
\usepackage{subcaption}
\usepackage{tikz}
\usetikzlibrary{decorations.markings}
\usepackage{biblatex}
\hypersetup{
    colorlinks=false,
    pdfborder={0 0 0},
}

\newtheorem{theorem}{Theorem}[section]
\newtheorem{lemma}[theorem]{Lemma}
\newtheorem{proposition}[theorem]{Proposition}
\newtheorem{corollary}[theorem]{Corollary}

\theoremstyle{definition}
\newtheorem{definition}[theorem]{Definition}
\newtheorem{examplex}[theorem]{Example}

\newenvironment{example}
  {\pushQED{\qed}\examplex}
  {\popQED\endexamplex}

\theoremstyle{remark}
\newtheorem{remark}[theorem]{Remark}
\newtheorem{remarks}[theorem]{Remarks}

\numberwithin{equation}{section}

 \allowdisplaybreaks

\newcommand{\1}[1]{{\mathbbm{1}\mkern -1.5mu}{\{#1\}}}

\newcommand{\R}{{\mathbb R}}

\newcommand{\Z}{{\mathbb Z}}
\newcommand{\N}{{\mathbb N}}

\newcommand{\ZP}{{\mathbb Z}_+}
\newcommand{\RP}{{\mathbb R}_+}

\DeclareMathOperator{\Exp}{\mathbb{E}}
\let\Pr\relax
\DeclareMathOperator{\Pr}{\mathbb{P}}

\DeclareMathOperator{\card}{\#}

\newcommand{\geo}[1]{\mathrm{Geom}_0 \left( {#1} \right)}

\newcommand{\eps}{\varepsilon}

\newcommand{\re}{{\mathrm{e}}}
\newcommand{\ud}{{\mathrm d}}

\newcommand{\cL}{{\mathcal L}}

\newcommand{\cP}{{\mathcal P}}

\newcommand{\cV}{{\mathcal V}}

\newcommand{\as}{\ \text{a.s.}}

\newcommand{\IP}{{\mathbb P}}
\newcommand{\IE}{{\mathbb E}}

\newcommand{\hsig}{{\hat \sigma}}

\newcommand{\xheaviside}{x_{\textup{H}}}

\newcommand{\stochdoml}{\precsim}
\newcommand{\stochdomg}{\succsim}

\newcommand{\cVF}{{\mathcal V}_{\mathrm{F}}}
\newcommand{\bbX}{{\mathbb X}}
\newcommand{\bbXB}{{\mathbb X}_{\mathrm{F}}}
\newcommand{\bbD}{{\mathbb D}}
\newcommand{\bbDB}{{\mathbb D}_{\mathrm{F}}}

 \newcommand{\abar}{{\overline \alpha}}
 \newcommand{\bbar}{{\overline \beta}}
  \newcommand{\oeta}{{\overline \eta}}
  \newcommand{\heta}{{\hat \eta}}
  \newcommand{\teta}{{\tilde \eta}}
  \newcommand{\tcL}{{\widetilde \cL}}
  
\makeatletter
\def\namedlabel#1#2{\begingroup  
    (#2)%
    \def\@currentlabel{#2}%
    \phantomsection\label{#1}\endgroup
}
\makeatother

\newlist{myenumi}{enumerate}{10}
\setlist[myenumi]{leftmargin=0pt, labelindent=\parindent, listparindent=\parindent, labelwidth=0pt, itemindent=!, itemsep=1pt, parsep=4pt}

\newlist{thmenumi}{enumerate}{10}
\setlist[thmenumi]{leftmargin=0pt, labelindent=\parindent, listparindent=\parindent, labelwidth=0pt, itemindent=!}

\title{Semi-infinite particle systems with exclusion interaction and heterogeneous jump rates}

\author{Mikhail Menshikov\footnote{\raggedright Department of Mathematical Sciences, Durham University, Durham, UK. \href{mailto:mikhail.menshikov@durham.ac.uk}{\texttt{mikhail.menshikov@durham.ac.uk}}, \href{mailto:andrew.wade@durham.ac.uk}{\texttt{andrew.wade@durham.ac.uk}}.} \and Serguei Popov\footnote{Centro de Matem\'atica, University of Porto, Porto, Portugal. \href{mailto:serguei.popov@fc.up.pt}{\texttt{serguei.popov@fc.up.pt}}.} \and Andrew Wade\footnotemark[1]}

\date{\today}

\begin{document}
\maketitle

\begin{abstract}
We study semi-infinite particle systems on the one-dimensional integer lattice,
where each particle performs a continuous-time nearest-neighbour random walk,
with jump rates intrinsic to each particle, subject to an exclusion interaction which
suppresses jumps that would lead to more than one particle occupying any site.
Under appropriate hypotheses on the jump rates (uniformly bounded rates is sufficient) and started from an initial condition that is a finite perturbation of the close-packed configuration, we give conditions under which the particles evolve as a single, semi-infinite ``stable cloud''. More precisely, we show that inter-particle separations converge to a product-geometric stationary distribution, and that the location of every particle obeys a strong law of large numbers with the same characteristic speed.  \end{abstract}

\medskip

\noindent
{\em Key words:}  
Exclusion process, infinite Jackson network, interacting particle systems, lattice Atlas model, asymptotic speeds, invariant measures, product-geometric distribution.

\medskip

\noindent
{\em AMS Subject Classification:} 60K35 (Primary), 60J27, 60K25, 90B22 (Secondary).

\section{Introduction and main results}

\subsection{Definition of the model}
\label{sec:model}

The exclusion process is a prototypical interacting-particle system, representing the dynamics of a  lattice gas
with hard-core interaction, originating with~\cite{spitzer,mgp}, and a
subject of ongoing interest in probability theory and
non-equilibrium statistical physics. The most well-studied versions of the model 
have homogeneous particles (see Section~\ref{sec:literature} for some comments on the literature).
Our interest here is in non-homogeneous systems. In the present paper,
 a sequel to the paper~\cite{mmpw} of the authors and V.A.\ Malyshev, which examined
\emph{finite} systems, we consider \emph{semi-infinite} systems. 
That is, we have particles enumerated left-to-right by $\N := \{1,2,3,\ldots\}$,
living  on the one-dimensional integer lattice $\Z$, performing continuous-time, nearest-neighbour random walks with exclusion interaction,
in which each particle possesses arbitrary finite positive jumps rates.  We describe the model more precisely.

The configuration space of the system is $\bbX$, given by
\begin{equation}
\label{eq:configuration-space}
\bbX  := \{ (x_1, x_2, \ldots ) \in \Z^{\N} : x_1 < x_2 < \cdots  \} .  \end{equation}
The dynamics are described by a time-homogeneous, continuous-time Markov process on~$\bbX$,
specified by finite, non-negative rate parameters $a_1,b_1, a_2, b_2, \ldots$. The $k$th particle (enumerated from the left)
attempts to make a nearest-neighbour jump to the left at rate~$a_k$. If, when the corresponding exponential clock rings,
the site to the left is unoccupied, the jump is executed and the particle moves, but if the destination is occupied by another particle, the attempted jump is suppressed and the particle does not move (this is the exclusion rule). 
Similarly, the $k$th particle attempts to jump to the right at rate $b_k$,
subject to the exclusion rule. 
The exclusion constraint ensures  that there can be at most one particle at any site of~$\Z$ at any given time,
and that the order of particles is preserved; in particular, the left-most particle is always the particle labelled~1.
See Figure~\ref{fig:particles} for a picture.
 
\begin{figure}[h]
\centering
\scalebox{0.85}{
 \begin{tikzpicture}[domain=0:1, scale=1.0]
\draw[dotted,<->] (0,0) -- (13,0);
\node at (13.4,0)       {$\Z$};
\draw[black,fill=white] (1,0) circle (.5ex);
\draw[black,fill=black] (2,0) circle (.5ex);
\node at (2,-0.6)       {\small $X_1 (t)$};
\draw[black,fill=white] (3,0) circle (.5ex);
\draw[black,fill=black] (4,0) circle (.5ex);
\node at (4,-0.6)       {\small $X_2 (t)$};
\draw[dotted] (2,0) -- (2,2);
\draw[dotted] (4,0) -- (4,2);
\draw[black,<->] (2,1.5) -- (4,1.5);
\node at (3,1.9)       {\small $1+\eta_1 (t)$};
\draw[black,fill=black] (5,0) circle (.5ex);
\node at (5,-0.6)       {\small $X_3 (t)$};
\draw[dotted] (5,0) -- (5,1.7);
\draw[black,<->] (4,1.5) -- (5,1.5);
\node at (5,1.9)       {\small $1+\eta_2 (t)$};
\draw[black,fill=white] (6,0) circle (.5ex);
\draw[black,fill=white] (7,0) circle (.5ex);
\draw[black,fill=white] (8,0) circle (.5ex);
\draw[black,fill=black] (9,0) circle (.5ex);
\node at (9,-0.6)       {\small $X_4 (t)$};
\draw[dotted] (9,0) -- (9,2);
\draw[black,<->] (5,1.5) -- (9,1.5);
\node at (7,1.9)       {\small $1+\eta_3 (t)$};
\draw[black,fill=black] (10,0) circle (.5ex);
\node at (10,-0.6)       {\small $X_5 (t)$};
\draw[black,fill=black] (11,0) circle (.5ex);
\node at (11,-0.6)       {\small $~~~~~~ X_6 (t) ~ \cdots$};
\draw[black,fill=white] (12,0) circle (.5ex);
\node at (1.3,0.6)       {\small rate $a_1$};
\node at (2.5,0.6)       {\small $b_1$};
\draw[black,->,>=stealth] (2,0) arc (30:141:0.58);
\draw[black,->,>=stealth] (4,0) arc (30:141:0.58);
\node at (3.5,0.6)       {\small $a_2$};
\node at (5.5,0.6)       {\small $b_3$};
\draw[black,->,>=stealth] (9,0) arc (30:141:0.58);
\node at (8.5,0.6)       {\small $a_4$};
\node at (11.5,0.6)       {\small $b_6$};
\draw[black,->,>=stealth] (2,0) arc (150:39:0.58);
\draw[black,->,>=stealth] (5,0) arc (150:39:0.58);
\draw[black,->,>=stealth] (11,0) arc (150:39:0.58);
\end{tikzpicture}}
\caption{Schematic of the model showing the leftmost $6$ particles, illustrating some of the main notation. Filled circles represent particles, empty circles represent unoccupied lattice sites, and
directed arcs represent admissible transitions, with exponential rates indicated.}
\label{fig:particles}
\end{figure}
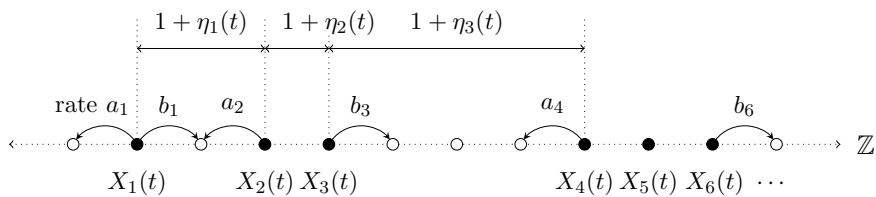

In the present paper, our main (but not only) interest will be in starting the system from
configurations that are approximately  close-packed. Define
\begin{equation}
\label{eq:XB-def}
    \bbXB := \Bigl\{ x \in \bbX : x_{k+1} - x_k = 1 \text{ for all but finitely many } k \in \N \Bigr\};
\end{equation}
we will refer to $\bbXB$ as the set of \emph{finite configurations}, because there are only finitely many empty sites (\emph{holes}) between particles.

Under some conditions on the rates $a_k, b_k$ that we will describe shortly, we will study dynamics of the
particle system started from $\bbXB$, and in particular focus
on the situation where the particles evolve as a single, semi-infinite ``stable cloud'', by which we mean that inter-particle separations are tight. In this case, we show that every particle obeys a strong law of large numbers with the same characteristic speed (negative or zero). These main results are presented in Section~\ref{sec:main-results}. First we address the question of conditions on the rates and initial configuration under which the  dynamics described above
give a well-defined Markov process.

We will often assume positivity of all the rates:
\begin{description}
\item\namedlabel{ass:positive-rates}{$\text{A}_0$} 
Suppose that $0 <a_k < \infty$ and $0 < b_k < \infty$ for all $k \in \N$.
\end{description}
The set $\bbXB \subset \bbX$ defined at~\eqref{eq:XB-def} is countable (whereas $\bbX$ is uncountable). The \emph{Heaviside} configuration 
is $\xheaviside := (0,1,2,3,\ldots) \in \bbXB$,
in which particles occupy $\ZP := \{ 0,1,2,\ldots \}$;
the other
elements of $\bbXB$ are finite perturbations of $\xheaviside$ and their translates. 
The following  rate condition (which ensures ``non-explosion'')
will provide an existence result in this case.

\begin{description}
\item\namedlabel{ass:a-inverse-sum}{$\text{A}_1$} 
Suppose that $\sum_{k \in \ZP} ( 1/a_k ) = \infty$.
\end{description}

Let $X (t) = (X_1 (t), X_2(t), \ldots  ) \in \bbX$ be the configuration of the Markov process
at time $t \in \RP := [0,\infty)$,
started from a fixed initial configuration $X(0) \in \bbX$.

\begin{proposition}[Existence of finitely-supported process]
\label{prop:existence-from-heaviside}
Suppose that~\eqref{ass:a-inverse-sum} holds, 
and 
 that~$X(0)\in\bbXB$. Then there exists a unique continuous-time Markov chain
 $(X(t))_{t \in \RP}$  on the countable state space $\bbXB$ (defined for all time $t \in \RP$)
 whose transition rates are as described above.
 \end{proposition}

 For initial conditions beyond $\bbXB$, we will assume the following:
\begin{description}
\item\namedlabel{ass:bounded-rates}{$\text{A}_2$} 
Suppose that $\sup_{k \in \N} a_k < \infty$ and $\sup_{k \in \N} b_k < \infty$.
\end{description}
Note that~\eqref{ass:bounded-rates} implies~\eqref{ass:a-inverse-sum},
but~\eqref{ass:a-inverse-sum} also permits $a_k \to \infty$ (for example $a_k = O(k)$).

 We will show (Proposition~\ref{prop:invariant-measure} below)
 that, assuming~\eqref{ass:bounded-rates}, one can construct
 $X(t)$ started from arbitrary $X(0) \in \bbX$,
 and, moreover, initial distributions that have a relative invariance property,
meaning that the distribution of the collection of inter-particle separations is stationary, while the whole system possesses a characteristic \emph{speed}. The stationary distribution for the inter-particle separations 
can be viewed as  an invariant measure (of which there may be many) 
of an \emph{infinite Jackson network}
 corresponding to the particle system. The Jackson network is helpful for both the construction
 of the process and its analysis. The Jackson connection, explained in detail in
 Section~\ref{sec:jackson}, is via the process of inter-particle separations; we introduce
 relevant notation next.
For $x = (x_1, x_2, \ldots) \in \bbX$,
 define $D_k : \bbX \to \ZP $, by 
\begin{equation}
    \label{eq:D-def}
D_k (x) := x_{k+1} -x_k -1, \text{ for } k \in \N.
\end{equation}
Set $\bbD := \ZP^\N$, and define $D : \bbX \to \bbD$
by $(D(x))_k = D_k (x)$. An example
is $D (\xheaviside) = 0 \in \bbD$. 
We define the process $\eta := ( \eta(t))_{t \in \RP}$ on $\bbD$ by $\eta (t) := D ( X(t) )$, $t \in \RP$. In other words, 
\begin{equation}
\label{eq:eta-def}
\eta_k (t) := X_{k+1}(t)-X_k(t)-1, 
\text{ for } k \in \N,
\end{equation}
denotes  the number of  {holes} 
between particles $k$ and $k+1$ at time~$t \in \RP$. An equivalent description of the system is captured by 
the
process
$\xi := ( \xi (t) )_{t \in \RP}$, where 
\begin{equation}
\label{eq:xi-def}
\xi (t) := (X_1 (t), \eta_1(t), \eta_2 (t), \ldots  ) \in \Z \times \ZP^\N .\end{equation}
It turns out that, under appropriate conditions,  $\eta = (\eta (t))_{t \in \RP} \in \bbD$, 
with coordinates given by~\eqref{eq:eta-def},  
is \emph{also} a continuous-time Markov process, which can  be represented via an \emph{infinite Jackson network} of queues of M/M/1 type, or, alternatively, as a \emph{zero range process} (we explain this is Section~\ref{sec:jackson}). 

Before describing our main results, let us 
indicate some well-known examples. Suppose that $X(0) \in \bbXB$,
so that existence is a consequence of
 Proposition~\ref{prop:existence-from-heaviside}. 
Since $\bbXB$ is countable, we can in this case understand the process $X(t)$ started from
an initial state in $\bbXB$ 
as a continuous-time Markov chain on a countable state space; under~\eqref{ass:positive-rates}
the Markov chain is irreducible on a subset of $\bbXB$, and so the standard notions of transience, recurrence, positive recurrence, and so on, apply. 
Since configurations in $\bbXB$ are ultimately close-packed to the right,
one has a uniform upper bound on $X_1 (t)$ (depending only on $X_1(0)$). Hence transience is equivalent to $\lim_{t \to \infty} X_1 (t) = -\infty$, a.s., and recurrence to $X_1 (t) = X_1 (0)$
for an a.s.-unbounded set of times $t$, for example.

\begin{example}[Homogeneous simple exclusion]
\label{ex:simple-exclusion}
Consider homogeneous rates $a_k \equiv a \in (0,\infty)$, $b_k \equiv b \in (0,\infty)$ for all $k \in \N$.
If $a = b$ this is  the \emph{symmetric simple exclusion process} (SSEP)
and if $a \neq b$ it is the  \emph{asymmetric simple exclusion process} (ASEP).
A classical result is that the Markov chain $X(t)$ on $\bbXB$
(with initial configuration $X(0) \in \bbXB$)
is positive recurrent  if $a <b$,
while if $a > b$ it is transient. 
As an example of our general results (see Example~\ref{ex:homogeneous-first} below),
we will see that for $a <b$,
$\eta(t)$ converges to a product-geometric stationary distribution,
while $\lim_{t \to \infty} t^{-1} X_k (t) =0$, a.s., for every $k \in \N$.
In the case $a=b$,
 a result of Arratia, Theorem~2 of~\cite[p.~386]{arratia83},
 says that 
 \begin{equation}
     \label{eq:arratia}
\lim_{t\to\infty}  \frac{X_1 (t)}{\sqrt{ t \log t}} = 1, \as
  \end{equation}
Arratia's result implies transience; an alternative approach to transience, which also treats some mixture models, can be found in 
Belitsky~\emph{et al.}~\cite{BFMP01}. 
\end{example}

In the next section we present our main results. 
We will return to various aspects of Example~\ref{ex:simple-exclusion},
and its generalizations, for illustration.
 
\subsection{Stability, strong law, and examples}
\label{sec:main-results}

In this section we present our main results, the necessary notation, and some examples;
proofs come later in the paper (see Section~\ref{sec:outline} for an outline). 
A central place in our analysis is occupied by the following linear system whose coefficients are the rate parameters $(a_i, b_i )_{i\in\N}$:
\begin{equation}
\label{eq:stable-traffic}
 (b_i +a_{i+1})\rho_i = a_i \rho_{i-1}+b_{i+1}\rho_{i+1}, \text{ for } i \in \N; ~ \rho_0 = 1.
\end{equation}
We call the system~\eqref{eq:stable-traffic} (always imposing the initial condition~$\rho_0 = 1$)
the \emph{stable traffic equation},
by analogy with the finite Jackson network setting (see~\cite{mmpw}), although in the infinite case the picture is richer.
\begin{lemma}[Solutions to the stable traffic equation]
\label{lem:solution-space}
Suppose that~\eqref{ass:positive-rates} holds. 
The set of solutions~$\rho = (\rho_k)_{k\in\ZP}$ to~\eqref{eq:stable-traffic} is the one-parameter family
\begin{equation}
\label{eq:rho-explicit}
 \rho = \alpha + v \beta
\end{equation}
for~$v\in \R$, where $\alpha = (\alpha_k)_{k \in \ZP}$
and $\beta = (\beta_k)_{k \in \ZP}$ are defined by 
\begin{align}
\label{eq:alpha-k-def}
 \alpha_0 & := 1, ~\text{and}~ \alpha_k := \frac{a_1\cdots a_k}{b_1\cdots b_k} \text{ for } k \in \N; \\
\label{eq:beta-k-def}
\beta_0 & := 0 ,  ~\text{and}~ \beta_k := \frac{1}{b_k}+\frac{a_k}{b_k b_{k-1}}+
  \cdots + \frac{a_k\cdots a_2}{b_k\cdots b_1}
   \text{ for }k\in \N.
\end{align}
Moreover, 
for every solution $\rho = \alpha + v \beta$, $v \in \R$,
\begin{equation}
\label{eq:v-difference}
v = b_{k+1}\rho_{k+1}-a_{k+1}\rho_k , \text{ for every } k \in \ZP.
\end{equation}
\end{lemma}
In the finite case~\cite{mmpw}, there is another boundary condition, so there is always a unique solution to the corresponding
traffic equation; the key result is that the
collection of inter-particle separations (which correspond to a Jackson queueing network, as described in Section~\ref{sec:jackson})
is stable if and only if that unique solution
satisfies $\rho_k <1$ for all $k \geq 1$. 

Our main interest is in the case where the particles in our semi-infinite
system constitute a single ``stable cloud'', and this
motivates considering the following subclass of solutions to the infinite
stable traffic equation~\eqref{eq:stable-traffic}.

\begin{definition}[Admissible solutions]
\label{def:admissible}
We say $\rho = (\rho_k)_{k\in\ZP}$ is an \emph{admissible} solution (to the stable traffic equation) if~$\rho$ satisfies~\eqref{eq:stable-traffic}, $\rho_0 = 1$, 
and $\rho_k \in (0,1)$ for all $k \in \N$. Let
\begin{equation}
\label{eq:admissible-speeds}
\cV := \{ v \in \R : \alpha + v \beta \text{ is an admissible solution to~\eqref{eq:stable-traffic}} \} .\end{equation}
\end{definition}
The  focus of the present paper is when~$\cV$ 
is non-empty, and when a particular member of~$\cV$ governs
the asymptotic dynamics of the particle system, exhibiting the
behaviour of a single ``stable cloud''. A complication
here (absent in the case of finite systems~\cite{mmpw}) is that
there may be multiple members of~$\cV$, and which one is relevant
for the dynamics can depend on the initial configuration, particularly once we permit initial configurations beyond~$\bbXB$.

\begin{remark}[Multiple clouds]
    \label{rem:non-admissible}
    A more general situation than we consider in the present paper is when the system can be decomposed into a single, semi-infinite stable cloud, plus a finite number of  finite stable clouds (to the left, necessarily). Such systems are described by solutions~$\rho$ to~\eqref{eq:stable-traffic} for which $\rho_k \geq 1$ for at least one $k \in \N$. We anticipate that a combination of the approach here and in~\cite{mmpw} (where finite systems are studied) can in this case be used to characterize the cloud decomposition in terms of~$\rho$ and to show that each (finite or infinite) cloud will satisfy a strong law of large numbers with its own intrinsic speed, ordered left-to-right.
    The situation with an infinite number of finite clouds   is also possible.
\end{remark}

Write $\geo{q}$ for the (shifted) geometric distribution on $\ZP$ with success parameter~$q \in (0,1]$, i.e., $\zeta \sim \geo{q}$ means that $\IP ( \zeta = n ) = (1-q)^n q$ for $n \in \ZP$. For an admissible solution~$\rho$ to~\eqref{eq:stable-traffic}, let $\nu_\rho$ be the product measure $\bigotimes_{k \in \N} \geo{ 1-\rho_k}$ on $\bbD = \ZP^\N$, i.e.,
\begin{equation}
\label{eq:df_measure_rho}
\nu_\rho ( m )
   = \prod_{k\in A} (1-\rho_k) \rho_k^{m_k}, \text{ for all  finite } A\subset \N \text{ and all } m = (m_k)_{k \in A}  \in \ZP^A.
\end{equation}

The following statement 
(based on results of~\cite{FF94}) shows that the $v \in \cV$ index invariant measures for the $\eta(t)$ component of the process $\xi(t)$ (i.e., the system seen from the leftmost particle),
with $v$ corresponding to the \emph{speed} of every particle in the particle cloud
when the system runs under stationary
law $\eta(t) \sim \nu_\rho$.  Since $\nu_\rho$ may be supported
on configurations with infinite total number of holes, 
existence of the process started from $\nu_\rho$ is, in general, not covered by Proposition~\ref{prop:existence-from-heaviside},
and the following statement provides existence under the 
bounded-rates condition~\eqref{ass:bounded-rates}. 

\begin{proposition}[Existence and invariant measures]
\label{prop:invariant-measure}
Suppose that~\eqref{ass:positive-rates} and~\eqref{ass:bounded-rates} hold. 
Then there exists a
Markov process $(\xi(t))_{t \in \RP}$ on $\bbX$ (defined for all time) with arbitrary 
 initial state $\xi(0) \in \bbX$, 
 possessing right-continuous sample paths, and dynamics as described 
 in Section~\ref{sec:model} above.
 
Moreover, 
suppose that~$v \in \cV$
with~$\rho = \alpha + v \beta$  the  corresponding
 admissible solution of~\eqref{eq:stable-traffic}. Then 
 if $\eta(0) \sim \nu_\rho$, for $\nu_\rho$ given by~\eqref{eq:df_measure_rho}, 
 and $\xi (0) = (X_1(0), \eta(0))$ for fixed $X_1(0) \in \Z$, it holds that
$\eta(t) \sim \nu_\rho$
 for all $t \in \RP$ (invariance), 
and 
\begin{equation}
\label{eq:invariant-speed}
\lim_{t \to \infty} \frac{X_1(t)}{t} = v , \as \end{equation}
\end{proposition}

When $\# \cV > 1$, Proposition~\ref{prop:invariant-measure} shows that there are multiple invariant measures associated with the particle system dynamics. The next result
identifies a minimal member of~$\cV$ of particular importance for our purposes.
Recall $\alpha_k, \beta_k$ defined at~\eqref{eq:alpha-k-def} and~\eqref{eq:beta-k-def}.

\begin{proposition}
\label{prop:admissible-solutions-first}
    Suppose that~\eqref{ass:positive-rates} holds. Then there always exists the limit 
\begin{equation}
\label{eq:v0-def}
v_0 := - \lim_{k\to\infty} \frac{\alpha_k}{\beta_k}
 = -\Big(\frac{1}{a_1}+\frac{b_1}{a_1a_2}
  + \frac{b_1b_2}{a_1a_2a_3}+ \cdots\Big)^{-1} \in (-a_1,0].
\end{equation}
If $\cV \neq 0$, then  $\inf \cV = v_0 \in \cV$.
Moreover,
 it holds that 
$v_0 \in \cV$ if and only if, for all $k \in \N$,
\begin{equation}
\label{eq:v0-admissible}
v_0 < \frac{b_1 \cdots b_k - a_1 \cdots a_k}{b_2 \cdots b_k + a_1 b_3 \cdots b_k + \cdots + a_1 \cdots a_{k-1}} ,
\end{equation}
where each of the $k$ terms in the denominator is a product of $k-1$ factors. 
\end{proposition}

In view of Proposition~\ref{prop:admissible-solutions-first}, the following terminology is appropriate.

\begin{definition}[Minimal solution]
\label{def:minimal-solution}
If $\cV \neq \emptyset$, then for $v_0 = \inf \cV$
we say that $\rho = \alpha + v_0 \beta$ is the \emph{minimal} admission solution to~\eqref{eq:stable-traffic} (or just ``minimal solution'', for short).
\end{definition}

It turns out for  initial configurations   in $\bbXB$, the minimal admissible solution (assuming $\cV \neq \emptyset$)
will take special significance. Intuitively, solutions with $v>0$ are not
accessible from a closely-packed configuration, as there is no space to the right for particles to escape, so only solutions with $v \leq 0$ remain. Either $v_0 = 0$, where only the zero-speed solution is accessible,
or else $v_0 <0$, but in that case, under condition~\eqref{ass:a-inverse-sum}, there can be at most one admissible solution (see Proposition~\ref{prop:admissible-solutions}\ref{prop:admissible-solutions-iii} below) which must also be the one corresponding to~$v_0$. A 
manifestation of the special role of~$v_0$ is the following strong law for speeds.

\begin{theorem}[Strong law of large numbers]
\label{thm:strong-law}
Suppose that~\eqref{ass:positive-rates} and~\eqref{ass:a-inverse-sum} hold, $\cV \neq \emptyset$, $X(0) \in \bbXB$,
and either (i)~$v_0 = 0$, or (ii)~$\bbar := \limsup_{k \to \infty} \beta_k = \infty$.
Then, for every $k \in \N$,  
    \begin{equation}
        \lim_{t \to \infty} \frac{X_k(t)}{t} = v_0, \as
    \end{equation}
\end{theorem}

Proposition~\ref{prop:admissible-solutions-first} shows that
to check whether $\cV \neq \emptyset$, it suffices to check $v_0 \in \cV$.

\begin{example}[Stable with zero speed]
Suppose that~\eqref{ass:positive-rates} holds, and that $v_0  = 0$.
Then, by Proposition~\ref{prop:admissible-solutions-first}, 
\begin{equation}
    \label{eq:0-stable}
0 = v_0 \in \cV \text{ if and only if }
b_1 \cdots b_k > a_1 \cdots a_k \text{ for every } k \in \N.
\end{equation}
The interpretation of~\eqref{eq:0-stable} is, roughly speaking,
that the particles, collectively, all want to travel to the right, which they cannot
do; hence the system
is stable with zero speed.
\end{example}
\begin{remark}[Dual random walk]
\label{rem:p0-v0}
We will obtain another probabilistic interpretation of~$v_0$
in terms of a dual random walk~$Q$.
This
random walk is most conveniently described as
a random walk of a \emph{customer} in a 
version of queueing system
associated to the process $\eta(t)$,
and we defer a description until Section~\ref{sec:jackson} below, where we
introduce the Jackson network representation. Then, assuming~\eqref{ass:positive-rates},  comparing~\eqref{eq:v0-def}
and~\eqref{eq:p0-def} below shows that $|v_0|/a_1$ is an escape probability
of the random walk~$Q$, and, in particular, it holds that 
    \begin{equation}
        \label{eq:v0-recurrnce}
    v_0 = 0 \text{ if and only if $Q$ is recurrent,} 
        \end{equation}
        where $Q$ is the customer walk from Definition~\ref{def:customer-random-walk} below.
\end{remark}

We will see below (Lemma~\ref{lem:stationary-mean}) that if $\sum_{k \in \ZP} \rho_k < \infty$, then the measure $\nu_\rho$ given by~\eqref{eq:df_measure_rho} is supported
on configurations~$\eta \in \bbD$ with $\sum_{k \in \N} \eta_k < \infty$.
Also, we will see (Proposition~\ref{prop:admissible-solutions}\ref{prop:admissible-solutions-v} and Lemma~\ref{lem:finitely-supported})
that condition~\eqref{ass:a-inverse-sum} ensures that
if there is an admissible solution~$\rho$ with $\sum_{k \in \ZP} \rho_k < \infty$,
then in fact $v_0 = 0 \in \cV$ and  $\sum_{k \in \ZP} \alpha_k < \infty$,
and $\nu_\alpha$ is the unique invariant measure supported on $\bbDB$.
The following result gives convergence.

\begin{theorem}[Convergence to stationarity: finite configurations]
\label{thm:finitely-supported-convergence}
 Suppose that~\eqref{ass:positive-rates} and~\eqref{ass:a-inverse-sum} hold,
 that $\cV \neq \emptyset$, $\sum_{k \in \ZP} \alpha_k  < \infty$, and $X(0) \in \bbXB$. Then, 
for all $m_1, m_2, \ldots \in \ZP$,
 \begin{equation}
     \label{eq:finitely-supported-convergence}
 \lim_{ t \to \infty} \IP \bigg[ \bigcap_{k \in \N} \{ \eta_k (t) = m_k \} \bigg] = \prod_{k \in \N} (1-\alpha_k) \alpha_k^{m_k}  . \end{equation}
\end{theorem}

The following corollary says, under the same hypotheses,
that the position $X_1$ of the leftmost particle is ergodic.

\begin{corollary}
\label{cor:finitely-supported-convergence}
 Suppose that~\eqref{ass:positive-rates} and~\eqref{ass:a-inverse-sum} hold,
 that $\cV \neq \emptyset$, $\sum_{k \in \ZP} \alpha_k  < \infty$, and $X(0) \in \bbXB$.
 Then, there is a probability measure $\pi_\alpha$ on $\Z$ such that, for every finite $A \subset \Z$,
\begin{equation}
\label{eq:X1-ergodic}
 \lim_{t \to \infty} \frac{1}{t} \int_0^t \1{X_1 (s) \in A} \ud s = \pi_\alpha (A), \as \end{equation}
\end{corollary}

More generally, we may have 
 the minimal solution having $\sum_{k \in \ZP} \rho_k = \infty$,
so the Markov chain result, Theorem~\ref{thm:finitely-supported-convergence}, does not apply. 
The following result permits $X(0) \in \bbXB$ (i.e., we start with a fixed 
 finite configuration), or else
$\eta(0)$ is chosen in accordance to some measure~$\nu$
 which is dominated
 by~$\nu_\rho$ defined in~\eqref{eq:df_measure_rho}
 (the definition of stochastic domination is given in
 Section~\ref{sec:stochastic-domination} below).
Then we have \emph{local convergence}, in the following sense.

\begin{theorem}[Convergence to stationarity: general configurations]
\label{thm:local-convergence}
 Suppose that~\eqref{ass:positive-rates} and~\eqref{ass:bounded-rates} hold, 
 and that $\cV   \neq \emptyset$.
 Denote by $\rho = \alpha+v_0\beta$ 
 the minimal solution. Suppose that either $X(0) \in \bbXB$, or $X_1(0) \in \Z$ and $\eta(0)$ is chosen in accordance to some measure~$\nu$
 which is dominated by~$\nu_\rho$.
 Then, 
for every finite $A \subset \N$
 and all $m_k \in \ZP$ ($k \in A$),
 \[ \lim_{ t \to \infty} \IP  \bigg[ \bigcap_{k \in A} \{ \eta_k (t) = m_k \} \bigg] = \prod_{k \in A} (1-\rho_k) \rho_k^{m_k}  . \]
\end{theorem}
\begin{remark}
In Theorem~\ref{thm:local-convergence}, we may have $\card \cV > 1$. Under~\eqref{ass:bounded-rates},
either $\sum_{k\in\ZP} \alpha_k < \infty$ and then $\cV = \{ 0\}$
is the only solution
(Proposition~\ref{prop:admissible-solutions}\ref{prop:admissible-solutions-vi}), so that we are back in the setting of Theorem~\ref{thm:finitely-supported-convergence}, or else
$\sum_{k\in\ZP} \alpha_k = \infty$, every other admissible $\rho$
must have $\sum_{k \in \ZP} \rho_k = \infty$ (see Lemma~\ref{lem:finitely-supported}),
and non-uniqueness can occur only when $v_0 = 0$ (Proposition~\ref{prop:admissible-solutions}\ref{prop:admissible-solutions-iv}) so non-minimal solutions have positive associated speed. The intuition of Theorem~\ref{thm:local-convergence} is that, 
started from a finite configuration, it is only the minimal solution that is accessible.
\end{remark}

In the case when~$\sum_{k\in\ZP} \alpha_k = \infty$, the following
consequence of Theorem~\ref{thm:local-convergence} contrasts with
Corollary~\ref{cor:finitely-supported-convergence}.

\begin{corollary}
\label{cor:infinite_mass_rho} 
 Suppose that~\eqref{ass:positive-rates} and~\eqref{ass:bounded-rates} hold,
 that $\cV \neq \emptyset$, $\sum_{k \in \ZP} \alpha_k  = \infty$, and $X(0) \in \bbXB$.
Then $X_1(t)\to -\infty$ in probability as $t \to \infty$.
\end{corollary}

We return to Example~\ref{ex:simple-exclusion}.
 
 \begin{example}[Homogeneous simple exclusion]
     \label{ex:homogeneous-first}
     As in Example~\ref{ex:simple-exclusion}, assume $a_k \equiv a \in (0,\infty)$ and $b_k \equiv b \in (0,\infty)$. Then, by~\eqref{eq:alpha-k-def} and~\eqref{eq:beta-k-def}, $\alpha_k = (a/b)^k$ and, if $a \neq b$, $\beta_k = (1 - (a/b)^k)/(b-a)$,
with $\beta_k = k/a$ if $a=b$. Let $X(0) \in \bbXB$.
\begin{itemize}
    \item 
If $a < b$, then $v_0 = 0$, and 
solutions $\rho_k = \alpha_k + v \beta_k$ are admissible for $v \in \cV = [ 0, b-a)$.
The minimal solution has $\sum_{k \in \ZP} \alpha_k<\infty$,
while the non-minimal solutions have $\lim_{k \to \infty} \rho_k = \frac{v}{b-a} \in (0,1)$. 
Condition~\eqref{ass:bounded-rates} is satisfied, and Theorem~\ref{thm:strong-law} says that $\lim_{t \to \infty} t^{-1} X_k (t) =0$, a.s., for every $k \in \N$,
while Theorem~\ref{thm:finitely-supported-convergence} yields convergence to a product-geometric stationary distribution for the particle separations. 
\item
If $a \geq b$, then one can check that $v_0 = b-a$ but the admissibility condition  fails (since $\alpha_k + v_0 \beta_k \equiv 1$), so that $\cV = \emptyset$. \qedhere
\end{itemize}
\end{example}

Another class of instructive examples
is obtained by perturbing the
critical homogeneous case from Example~\ref{ex:simple-exclusion} by modifying the rates only of the \emph{leftmost} particle. 
As in the finite case~\cite{mmpw}, we call the leftmost particle the ``dog'',
and all the rest ``sheep''.

\begin{example}[Dog and sheep; lattice Atlas model]
\label{ex:lattice-atlas-first}
Suppose that 
\begin{equation}
    \label{eq:dog-sheep-rates}
0 < a_1 = a < b_1 = b < \infty, \text{ and, for all $k \geq 2$, } a_k \equiv b_k \equiv c \in (0,\infty).
\end{equation}
Let $X(0) \in \bbXB$. 
Assuming~\eqref{eq:dog-sheep-rates}, we have $\alpha_k = a/b \in (0,1)$, $\beta_k = \frac{k-1}{c} + \frac{1}{b}$, and 
$v_0 = 0$, $\cV = \{ 0 \}$, and the unique admissible solution is $\rho = \alpha$.
Condition~\eqref{ass:bounded-rates} is satisfied, and Theorem~\ref{thm:strong-law} says that $\lim_{t \to \infty} t^{-1} X_k (t) =0$, a.s., for every $k \in \N$. In this case $\sum_{k \in \ZP} \alpha_k = \infty$,
so Theorem~\ref{thm:finitely-supported-convergence} does not apply, but Theorem~\ref{thm:local-convergence}
yields local convergence of the particle separations to the invariant measure~$\nu_\alpha$ given by $\bigotimes_{k \in \N} \geo{ \frac{b-a}{b} }$.
This is a lattice relative of the continuum \emph{Atlas model}~\cite{bfk,ipbkf};
we refer to~\cite{mmpw} for a discussion of some links between
lattice and continuum models (in the finite setting).
\end{example}

We close this section with one further result about the behaviour of the
leftmost particle (the ``dog'') from Example~\ref{ex:lattice-atlas-first}; we saw already that $\lim_{t \to \infty} t^{-1} X_1 (t) =0$, a.s., i.e., the dog's speed is always zero.
On the other hand,  
Corollary~\ref{cor:infinite_mass_rho} 
says that $X_1 (t) \to -\infty$ in probability.
The following, stronger, result shows that the dog particle is 
transient to the left with approximately diffusive 
rate of escape; cf.~Arratia's result~\eqref{eq:arratia}
in the case $a=b$.

\begin{theorem}
\label{thm:dog_repelled}
Suppose that~\eqref{eq:dog-sheep-rates} holds with $0 < a < b = c = 1$, and that $X(0) \in \bbXB$.
Then, for some $C \in (0,\infty)$ and any~$\delta>0$ 
it holds that, a.s., for all $t \in \RP$ sufficiently large,
\begin{equation}
\label{eq_dog_repelled}
C t^{\frac{1}{2}}\leq  -X_1(t) \leq t^{\frac{1}{2}+\delta}.
\end{equation}
\end{theorem}
\begin{remark}
    \label{rem:dog_repelled}
    Theorem~\ref{thm:dog_repelled}
    could be extended 
    to a larger class of situations, such as
    several (but finitely many) ``dogs'' with more general rates, provided that the system
    nevertheless has $\cV \neq \emptyset$;
essentially the same proof would work,
as indicated in Remark~\ref{rem:dog_repelled_proof} below.
\end{remark}

We plan, in future work, to examine finer asymptotics of the leftmost particle, and consider in greater generality questions of recurrence and transience, for example.
 
\subsection{Discussion and related literature}
\label{sec:literature}

In the earlier paper~\cite{mmpw}, Malyshev and the present
authors
studied finite particle systems, and used classical results of Goodman \& Massey on partial stability
for Jackson networks~\cite{GM84}
to obtain a decomposition of the system (determined   by the rate parameters) into maximal stable subsystems (``clouds''). Each stable cloud has its own product-geometric limit distribution for the inter-particle separations, and its own asymptotic speed. We refer to~\cite{mmpw} for further results, and also a discussion of the relation to adjacent continuum (diffusion) models~\cite{bfk,ipbkf}.

The process $\eta(t)$ of inter-particle separations
in the exclusion process is a particular case
of the \emph{zero-range} process. 
Since the 1990s, there has been a great deal of
study of zero-range and exclusion processes with disordered 
environments, including heterogeneous but deterministic,  or randomly chosen,
rates;
in the exclusion process, rates attached to particles (as in our case)
correspond to rates attached to \emph{sites} in the zero-range process.
From the earliest work, it was recognized that 
there may be a rich landscape of invariant measures
(including product-geometric measures), and influence of invariant measures on
dynamics, either locally or globally,
is mediated through  densities of initial configurations. We refer to~\cite{evans96,evans97,kf,jl,fs,afgl,bfl,bj,bmrs2015,bmrs2017,bmrs2020a,bmrs2020b} and references therein for
extensive results on   phenomena concerning convergence to equilibrium and hydrodynamic limits.
 
For  infinite, heterogeneous exclusion systems,
the majority of the literature addresses doubly-infinite systems (i.e., particles enumerated by~$\Z$, rather than $\ZP$). 
For settings  in that context parallel to ours, we mention in particular
work of 
Benjamini~\emph{et al.}~\cite{bfl}
and a more recent programme of Bahadoran~\emph{et al.}~\cite{bmrs2015,bmrs2017,bmrs2020a,bmrs2020b}. 
These works display many interesting phenomena,  some similar to those presented here, including a family of invariant measures parametrized by density and possessing an associated speed.
These papers (and others cited therein) consider, in some cases, a more general random walk kernel than our nearest-neighbour walk,
but their ``site-wise disorder'' is a (random or nonrandom) environment 
 less general than ours: for example, in~\cite{bmrs2015}, 
  in our notation, it is assumed that 
$a_{k+1} = p \alpha (k)$  and $b_k = q \alpha (k)$, 
where $p + q =1$ and $p > q$, so that the underlying skeleton random walk is homogeneous, and the inhomogeneous rates are  such that
$c < \alpha(k) \leq 1$ for all $k$ and some $0 < c <1$.
Results in~\cite{bfl,bmrs2015}
give sufficient ``density'' conditions  on the initial configuration to guarantee  convergence 
to an invariant measure governed by the ``slow'' vertices. Subsequent results concern higher dimensions~\cite{bmrs2017} and hydrodynamic limits~\cite{bmrs2017,bmrs2020a}. We emphasize that the main distinction between our results and most of this body of work is that our interest is in semi-infinite systems, particularly the behaviour of the leftmost particle and its neighbourhood, i.e., we examine extremal behaviour, rather than bulk behaviour. Since the leftmost particle has no influence from particles to its left, behaviour can be quite different than in the doubly-infinite setting.

We wish to draw particular attention to~\cite{bfl},  
which does include results on semi-infinite systems (although their main interest is in systems indexed by $\Z$, as described above).
However, 
in the setting which comes closest to ours, in~\cite{bfl} it is assumed that all particles have a uniformly negative intrinsic drift, and it is shown, roughly speaking, that the system is transient to the left at the speed of the slowest particles.

\subsection{Outline of the paper}
\label{sec:outline}

The rest of the paper builds to the proofs of the results presented in Section~\ref{sec:main-results}. 
In Section~\ref{sec:jackson} we introduce the Jackson network associated with the
particle system, which provides access to some useful terminology and intuition.
Here we give the proofs of our results on existence and
invariant measures (Propositions~\ref{prop:existence-from-heaviside} and~\ref{prop:invariant-measure}), and introduce a dual
 random walk, the customer random walk, and its properties. Section~\ref{sec:admissible} turns to a discussion of the admissible
 solutions to~\eqref{eq:stable-traffic} and the structure of the set~$\cV$ from Definition~\ref{def:admissible}. Here we give proofs of Lemma~\ref{lem:solution-space} and Proposition~\ref{prop:admissible-solutions-first}.
 Section~\ref{sec:invariant-measures} examines the role of the
 invariant measures~$\nu_\rho$ defined at~\eqref{eq:df_measure_rho},
 and the significance of cases in which~$\nu_\rho$ is supported
 on finite or on infinite configurations. The concept of second-class customers
 furnishes some useful stochastic domination results, and then we present
 comparison results with \emph{finite} particle systems. These tools enable us to present, in Section~\ref{sec:invariant-measures}, proofs of the convergence
 results Theorems~\ref{thm:finitely-supported-convergence} and~\ref{thm:local-convergence}, and their corollaries for the behaviour of the left-most particle,
 Corollaries~\ref{cor:finitely-supported-convergence} and~\ref{cor:infinite_mass_rho}. Finally, Section~\ref{sec:leftmost-particle}
 presents the proof of the strong law, Theorem~\ref{thm:strong-law}, and the asymptotics for the dog and sheep example, Theorem~\ref{thm:dog_repelled}.

\section{Infinite Jackson network and auxiliary random walk}
\label{sec:jackson}

\subsection{Jackson network representation}
\label{sec:jackson-terminology}

Jackson networks of finitely many queues have been extensively studied (see e.g.~\cite[Chs.~2~\&~7]{ChenYao} or~\cite[Ch.~1]{serfozo});   infinitely many queues have also received some attention~\cite{khmelev,ks,FF94,bfl}.
For us, the Jackson network serves as a 
framework in which to construct and describe (an enriched version of) the process $\eta$, as we describe in this section;
the link between exclusion processes and Jackson networks is well known,
going back at least to~\cite{kipnis} (see~\cite{mmpw} for further literature). 

In Section~\ref{sec:existence}
we use this approach to provide the construction of the process~$\eta$,  and hence~$\xi$, 
and give the proof of the existence and invariance results, Propositions~\ref{prop:existence-from-heaviside} and~\ref{prop:invariant-measure}. Then in Section~\ref{sec:customer-walk} we use the
Jackson representation to describe the \emph{customer random walk}, which serves as a useful tool
in our analysis.

First, in some generality, consider a countably infinite system of queues, enumerated by~$\N$. The parameters of the system are arrival rates $\lambda = (\lambda_i)_{i \in \N}$,
service rates $\mu = (\mu_i)_{i \in \N}$ and $P = (P_{ij})_{i,j\in\N}$, a sub-stochastic routing matrix.
Exogenous customers entering the system arrive at queue $i \in \N$ via an independent Poisson process of rate $\lambda_i \in \RP$.
Queue $i \in \N$ serves customers
at exponential rate $\mu_i \in \RP$
(i.e., in Kendall's notation, we have 
an $M/M/1$ queue at each server).
 Once a customer at queue $i$ is served, it is routed to a queue~$j$
with probability $P_{ij}$, while with probability $\delta_i := 1 - \sum_{j \in \N} P_{ij}$ the customer departs from the system.

Provided $\sum_{i \in \N} \lambda_i >0$ and $\sum_{i \in \N} \delta_i >0$, customers both enter and leave the system, and
it is called an \emph{open Jackson network}. We assume that every queue can be \emph{filled}, 
meaning that, for every $i \in \N$, there is a $j \in \N$ and $k \in \ZP$ for which $\lambda_j > 0$ and $(P^k)_{ji}>0$,
and that every queue can be \emph{drained}, meaning that,  for every $i \in \N$, there is a $j \in \N$ and $k \in \ZP$ for which $\delta_j > 0$ and $(P^k)_{ij}>0$.

We now explain the Jackson network connection to the particle system
model as described in Section~\ref{sec:model}; we use the notation introduced there.
Recall from~\eqref{eq:eta-def} that
$\eta_k (t) = X_{k+1} (t) - X_k (t) -1$,
the number of holes
between consecutive particles at time $t\in\RP$. Define
\begin{equation}
\label{eq:mu-def}
\mu_i := b_i + a_{i+1}, \text{ for }
i \in \N, \end{equation}
and
\begin{equation}
\label{eq:lambda-def}
\lambda_1 := a_1,   \text{ and } \lambda_i := 0 \text{ for } i \geq 2.
\end{equation} 
Also define the  matrix $P:= (P_{i,j})_{i,j \in \N}$ by
\begin{equation}
\label{eq:P-def}
\begin{split}
P_{i,i-1} & := \frac{b_i}{\mu_i} = \frac{b_i}{b_i+a_{i+1}}, \text{ for } i \geq 2; \\
P_{i,i+1} & := \frac{a_{i+1}}{\mu_i} = \frac{a_{i+1}}{b_i+a_{i+1}}, \text{ for } i \geq 1;
\end{split}
\end{equation}
and $P_{i,j}:=0$ for all $i,j$ with $\vert i-j\vert \neq 1$.

The process $\eta = (\eta_{i})_{i \in \N}$ can, under appropriate
conditions on the rates and initial states, 
be realised as the queue-length process for a corresponding Jackson network, namely,
the Jackson network with parameters $\lambda, \mu$ and $P$ given as functions of $(a_i,b_i)_{i \in \N}$ through formulas~\eqref{eq:mu-def}, \eqref{eq:lambda-def} and~\eqref{eq:P-def}. 

Putting aside questions of existence
for the moment,
which will be addressed in Section~\ref{sec:existence}, let us explain the correspondence between the queueing process and the particle system, as this will allow us to use two parallel
lenses to study our processes. 
In the queueing network, the customers waiting at queue~$k$
are equal in number to the unoccupied sites between particles $k$ and $k+1$ in the particle system. 
Exogenous customers enter the queueing network only when the leftmost particle jumps to the
left (rate $a_1$). Customers at queue~$k$ are served if either particle~$k$ jumps right
(rate $b_k$) or if particle $k+1$ jumps left (rate $a_{k+1}$). If  
particle~$k$ jumps right, the customer 
is routed to queue $k-1$ (if $k \geq 2$, at rate $b_k = \mu_k P_{k,k-1}$) or leaves the system (if $k=1$).
If particle $k+1$ jumps left, the customer is routed to queue $k+1$ (at rate $a_{k+1} = \mu_k P_{k, k+1}$). 
Customers leave the system only when the leftmost particle jumps to the right (rate $\mu_1 \delta_1 = b_1$).

It is sometimes useful to distinguish customers as they enter and leave the network.
At time $0$, we suppose that each queue is occupied by at most finitely many customers, i.e., $\eta(0) \in \bbD$; these
we call \emph{endogenous} customers (there are at most countably many), and we enumerate them in some arbitrary order. Customers that enter the system during time $(0,\infty)$ a.s.~do so at distinct times and with only finitely many arrivals in every compact time interval; these we call \emph{exogenous} customers, and we enumerate them in order of increasing arrival time. 

Before addressing existence of the process~$\eta$, we describe, first informally
but in a way that can be made rigorous once existence is in hand,
how one can construct~$\xi$
given the queueing process. We  would like to preserve the 
intuitive picture that 
 $X_1$ moves to the left only when exogenous customers enter the system, and $X_1$ moves to the right only when customers exit the system due to service at queue~1.
Take the construction of the Jackson network as given, and  introduce the following associated
counting processes.
  \begin{itemize}
     \item Let $E^{\rightarrow1}(t)$ be the number of exogenous customers that enter the system during $(0,t]$.
     \item Let $E^{\leftarrow1}(t)$ be the number of customers that depart the system during $(0,t]$.
 \end{itemize}
The superscript~$1$ indicates that entry and exit is possible only via queue~$1$.
Then we declare
 \begin{equation}
     \label{eq:X1-via-counting}
 X_1 (t) - X_1 (0) :=    E^{\leftarrow1}(t)  - E^{\rightarrow1}(t) = -M(t) , \text{ for all } t \in \RP, 
  \end{equation}
where $M(t) := E^{\rightarrow1}(t) - E^{\leftarrow1}(t)$ is the net change in total occupancy
of the queueing system during time $[0,t]$. If $T(t) := \sum_{k \in \N} \eta_k (t) \in \ZP \cup \{ \infty \}$ counts the total number of customers in the system, then since customers enter at bounded rate, $T(0) < \infty$ implies that $T(t) < \infty$ for all $t \in \RP$, and so 
$M(t) = T(t) - T(0)$ whenever $T(0) < \infty$.
Consequently,
 \begin{equation}
     \label{eq:conservation}
     X_1 (t) - X_1 (0) = - \sum_{k \in \N} \bigl( \eta_k (t) - \eta_k (0) \bigr), \text{ whenever } \sum_{k \in \N} \eta_k (0) < \infty.
 \end{equation}
 Our aim, therefore, is to use either~\eqref{eq:X1-via-counting}
 or~\eqref{eq:conservation} to define $X_1$ from $\eta$, and hence to give the full process~$\xi$.
There are some obstacles to this scheme (and hence to existence),
most significantly the possibility of \emph{explosion} of customers through the network,
meaning that customers are routed successively through increasing queues so fast that they ``escape to infinity'' in finite time. In the particle system, one has a choice how to accommodate this; for example, there could be a corresponding flux of particles ``from infinity'' entering the system (see
also Remark~\ref{rem:unbounded-rates} below). 
Under our rate hypotheses~\eqref{ass:a-inverse-sum} or~\eqref{ass:bounded-rates}, this explosion is excluded, as we explain in Section~\ref{sec:existence}.

\begin{remark}
\label{rem:service}
    The correspondence between the Jackson network and the particle system 
    is through the lengths of the queues in the queueing network, i.e., we can imagine customers are indistinguishable.
    One can enrich the queueing process by distinguishing customers, and/or retaining information about the time
    elapsed since their arrival at the present queue, and enforcing any priority service regime, such as FIFO (first-in, first-out). For definiteness, we can think of the service regime as being FIFO, but this is not important for our results.
\end{remark}

\subsection{Existence of the Markov process}
\label{sec:existence}

Consider
the configurations for $\eta \in \bbD$ with finitely many holes (or, equivalently,
finitely many occupied queues)
\begin{equation}
\label{eq:D-F-def}
\bbDB := \Bigl\{ \eta \in \bbD : \sum_{k \in \N} \eta_k < \infty \Bigr\}.  \end{equation}
In terms of the particle   configuration $x \in \bbX$ and the function $D$ defined at~\eqref{eq:D-def}, $D(x) \in \bbDB$ if and only if $x \in \bbXB$ as defined at~\eqref{eq:XB-def}.
Note that $\bbDB$ is countable.

 \begin{proof}[Proof of Proposition~\ref{prop:existence-from-heaviside}]
Take $\eta(0) \in \bbDB$ and $X(0) = (X_1(0) , \eta(0))$ for arbitrary $X_1(0) \in \Z$.
The process~$X (t) = (X(t))_{t \in \RP}$
can be realized
as a Markov chain on the augmented countable state space ${\overline \bbX}_\mathrm{F} := \bbXB \cup \{ \partial \}$
where the state $\partial$ accommodates potential explosion.
We claim that condition~\eqref{ass:a-inverse-sum}
ensures that, a.s., explosion does not occur. 
To see this, set $\bbXB^{(1)} := \Z \times \{ \xheaviside \}$, and for $k \geq 2$ write 
\[ \bbXB^{(k)} := \{ x \in \bbX : D_{k-1} (x) > 0 \text{ but } D_j (x) = 0 \text{ for all } j \geq k \} ,\]
those states for which all particles from particle~$k$ onwards
are tightly packed,
or, equivalently, queue~$k-1$ is the rightmost occupied queue.
Then $\bbXB = \cup_{k \in \N} \bbXB^{(k)}$
is a partition of $\bbXB$ in which the transition rates of $X$ restricted to
$\cup_{k = 1}^{m} \bbXB^{(k)}$ are bounded, for every $m \in \ZP$,
by, say,  $\sum_{k=1}^{m} | a_k + b_k | < \infty$.

At this point it is convenient, and no loss of generality, to assume $\eta (0) = 0$ so $X(0) \in \bbXB^{(1)}$.
If $\sigma_k := \inf \{ t \in \RP : X (t) \in \bbXB^{(k)} \}$, 
and $\sigma_\infty := \lim_{k \to \infty} \sigma_k \in [0,\infty]$ (the limit
exists by monotonicity)
sufficient for non-explosion is
thus to prove that $\Pr ( \sigma_\infty = \infty ) = 1$.
This we can achieve by comparison with a pure birth process
(see e.g.~\cite[pp.~19--20]{anderson}). 
Indeed, the increments $\sigma_{k+1} - \sigma_k$ dominate a sequence of independent exponential
$a_{k}$ random variables,
 since whenever queue~$k-1$ is occupied, queue $k$ receives a customer at rate~$a_{k}$.
 Hence, by monotone convergence,
 $\Exp \sigma_\infty = \sum_{k \in \N} \Exp ( \sigma_{k+1} - \sigma_k) \geq  \sum_{k \in \N} 1/a_k$,
 and~\eqref{ass:a-inverse-sum} ensures this sum diverges, and thus $\Pr ( \sigma_\infty = \infty ) = 1$.
 \end{proof}
 
We next give a proof of Proposition~\ref{prop:invariant-measure}.
Recall the definition of the product-geometric measure $\nu_\rho$ on $\bbD = \ZP^\N$
from~\eqref{eq:df_measure_rho}.

\begin{proof}[Proof of Proposition~\ref{prop:invariant-measure}]
For construction of the process and verification of invariance
of the measures $\nu_\rho$ for admissible~$\rho$,
we draw heavily on~\cite{FF94};
a related approach can be found in Appendix~B of~\cite{bmrs2017}, and 
see also~\cite{ferrari92,andjel82,bfl} for further remarks and references.

First, we indicate how the existence for the inhomogeneous semi-infinite exclusion process with bounded rates
follows from a standard Harris graphical construction~\cite{harris}. We can formally
define the queue process~$\eta$ on $\overline \bbD := \{ \infty \} \times \bbD$ by introducing a ``sink queue''
labelled by~$0$ which always holds an infinite number of customers,
 to accommodate customers going in and out of the system (in our case, the reservoir of
 holes to the left of the leftmost particle). 
 In coordinates, that is $\oeta_0 := \infty$ and $\oeta_k = \eta_k$ for $k \in \N$. 
 The generator of the process $\oeta := (\infty,\eta) \in \overline\bbD$ is then~$\cL$ acting on  local functions~$f : \overline\bbD \to \RP$ via
\begin{align}
\label{generator_LL}
 \cL f(\oeta)  = & 
    \sum_{k \in \N} \1{\oeta_k \neq 0 }\Bigl(a_{k+1} \bigl(f(\oeta^{k,k+1})-f(\oeta)\bigr)
    + b_k \bigl(f(\oeta^{k,k-1})-f(\oeta)\bigr)\Bigr)\nonumber\\
    & {}\quad + a_1 \bigl(f(\oeta^{0,1})-f(\oeta)\bigr),
\end{align}
 where $\oeta^{x,y} := (\oeta_z^{x,y})_{z \in \ZP}$ is
 the modification of $\oeta$ with a customer removed from queue~$x$ and added to queue~$y$, i.e.,
\[
  \oeta_z^{x,y} :=  
 \begin{cases}
   \eta_z, & z\notin \{ x,y \},\\
   \eta_x-1, & z=x,\\
   \eta_y+1, & z=y,
 \end{cases}
\]
with the convention $\infty \pm 1 := \infty$.

The graphical construction
is erected on an array of independent marked Poisson processes
associated with each queue, so the Poisson process labelled by $k \in \ZP$ has left-pointing arrows arriving at rate $b_k$ (apart from queue~$0$, which has none) and right-pointing arrows at rate $a_{k+1}$.
 For the case of uniformly bounded rates, fix some time $t \in \RP$, and try to construct the process up to time~$t$. Since rates are uniformly bounded, there is a positive probability at least $\re^{- ct}$, $c>0$, uniformly in $k \in \ZP$, that the Poisson process labelled $k$ has no arrivals in time $[0,t]$. Hence, a.s., there are infinitely many $k \in \ZP$ for which the corresponding Poisson process has no arrivals, and hence the construction of the process can be reduced to construction on (countably many) finite pieces. This completes the construction of $\eta$
(by forgetting the constant first coordinate of $\oeta$)
and the proof of invariance of $\nu_\rho$ for each admissible~$\rho$.
Then we can define~$X_1$ through~\eqref{eq:X1-via-counting}, and this gives
the construction of the particle system process~$\xi$.

Suppose that $v \in \cV$ with $\rho = \alpha + v \beta$ the corresponding
admissible solution of~\eqref{eq:stable-traffic}. 
For verification of invariance of $\nu_\rho$ defined at~\eqref{eq:df_measure_rho}, we essentially follow~\cite{FF94}; for this we   construct (a candidate for) the reverse process $\teta = (\teta(t))_{t \in \RP}$
with respect to the measure~$\nu_\rho$. It is again a process
with the state space $\overline\bbD$ with the queue~$0$ acting as an infinite reservoir
(i.e., always with infinite customers); for $k\in \N$,
customers are routed from queue $k$ to queue $k-1$ at rate $a_k \rho_{k-1} / \rho_{k}$
and from queue $k$ to queue $k+1$ at rate $b_{k+1} \rho_{k+1} / \rho_k$.
Let us denote its generator, formally defined analogously 
to~\eqref{generator_LL}, by $\tcL$:
\begin{align*}
 \tcL f(\teta)  = & 
    \sum_{k \in \N} \1{\teta_k \neq 0 }\biggl(\frac{b_{k+1}\rho_{k+1}}{\rho_{k}} \bigl(f(\teta^{k,k+1})-f(\teta)\bigr)
    + \frac{a_k\rho_{k-1}}{\rho_k} \bigl(f(\teta^{k,k-1})-f(\teta)\bigr)\biggr)\nonumber\\
    & {}\quad + b_1 \rho_1 \bigl(f(\teta^{0,1})-f(\teta)\bigr).
\end{align*}
To ensure the existence of the reverse process, we need to verify
that the above transition rates are bounded (this is, essentially,
condition~(9) of Theorem~1 of~\cite{FF94}). To do this, 
first suppose that $\liminf_{k\to\infty}\rho_k>0$.
Since $0 < \rho_k <1$ for all $k \in \N$, this means that $\rho_k \in [c,1]$ for some $c>0$
and all $k \in \N$, and then~\eqref{ass:bounded-rates} implies the rates in
the specification of $\tcL$ are uniformly bounded. On the other hand, observe 
that~\eqref{eq:v-difference} implies that
\begin{equation}
    \label{eq:rate-reversal}
 \frac{\rho_{k+1}}{\rho_k} = \frac{a_{k+1}}{b_{k+1}}
            + \frac{v}{\rho_k b_{k+1}}.
\end{equation}
If we assume that $\liminf_{k\to\infty}\rho_k=0$, then it follows from
Proposition~\ref{prop:admissible-solutions}\ref{prop:admissible-solutions-vi} below, and~\eqref{ass:bounded-rates},
that we must have $v = v_0=0$, and so~\eqref{eq:rate-reversal} shows that
the   process $\hat \eta$ is the same as the original process $\oeta$,
and then boundedness of rates is direct from~\eqref{ass:bounded-rates}.

The next step is to verify that $\oeta$ and $\teta$ 
are indeed reverse of one another with respect to~$\nu_\rho$.
This amounts to checking that 
\begin{equation}
\label{reversibility_generators}
 \int f\cL g \, d\nu_\rho = \int g{\tcL} f \, d\nu_\rho
\end{equation}
for all local functions $f,g$. We do not include the calculation
here since it was done in Proposition~1 of~\cite{FF94} in a more
general case. Now, it only remains to note that~\eqref{reversibility_generators}
implies that~$\nu_\rho$ is invariant both for~$\eta$
(since inserting $f\equiv 1$ to~\eqref{reversibility_generators}
yields $\int \cL g \, d\nu_\rho = 0$)
and~$\teta$ (now, use~\eqref{reversibility_generators}
with $g\equiv 1$). 

Finally, we verify the asymptotic speed result stated at~\eqref{eq:invariant-speed}.
Recall the representation~\eqref{eq:X1-via-counting}.
By time-reversal, it holds that, under the invariant measure $\nu_\rho$,
 the exit process
of~$\eta$ equals in distribution to the entrance process of~$\teta$.
In other words, $E^{\leftarrow1}(t)$ is a homogeneous Poisson process of rate
$b_1 \rho_1$.
On the other hand, $E^{\rightarrow1}(t)$ is a homogeneous Poisson process of rate
$a_1$. By the strong law for the Poisson process, $\lim_{t \to \infty} t^{-1}E^{\leftarrow1}(t) = b_1 \rho_1$ and $\lim_{t \to \infty} t^{-1}E^{\rightarrow1}(t) = a_1$, a.s., and hence, by~\eqref{eq:X1-via-counting},
we obtain, a.s., $\lim_{t \to \infty} t^{-1} X_1 (t) = b_1 \rho_1 - a_1 = v$, by~\eqref{eq:v-difference}.
This establishes~\eqref{eq:invariant-speed} and ends the proof.
\end{proof}
 
\begin{remark}
\label{rem:unbounded-rates}
In the present paper, we do not consider cases that fall outside
those covered by one or other of the existence results,
Propositions~\ref{prop:existence-from-heaviside} and~\ref{prop:invariant-measure}.
For example, starting from~$\bbXB$,
if one has $\sum_{k \in \N} a_k^{-1} < \infty$ 
it seems natural to declare that customers exit the system ``to infinity'' at a constant rate,
and then one has some choice about whether to correct the ``speed'' of the leftmost particle $X_1$.
As an example, consider $1=b_1>a_1=a$ and 
$a_k=b_k$ for $k\geq 2$ with $h:=\sum_{k \in \N} a_k^{-1}<\infty$. Then we see that
$\alpha_k=a$ and $\beta_k=a_1^{-1}+\cdots+a_k^{-1}$,
meaning that $v_0 = -a/h$ and $\rho = \alpha-(a/h)\beta$ is admissible.
A  direct calculation
implies that any finite cloud of leftmost particles 
in isolation 
would, however, go
to the right. 
Exploring such (somewhat pathological) cases is beyond the scope of this 
paper.
\end{remark}

\subsection{Customer random walk}
\label{sec:customer-walk}

Apart from providing access to a wealth of theory and intuition from queueing networks,
the Jackson representation enables us to define naturally a useful auxiliary random walk
whose properties shed light on the particle system.
From now on, when we use the terminology ``customer'', it is to be understood in the context 
of the queueing system described in Section~\ref{sec:jackson-terminology}.
At this point it is convenient
to distinguish one customer from another (cf.~Remark~\ref{rem:service}).
Since $\eta(0) \in \bbD$ consists of (at most) countably many customers,
and finitely many new customers enter the system by time $t \in \RP$,
there are countably many customers involved over time~$\RP$. We enumerate endogenous
 customers arbitrarily and then subsequent (i.e., exogenous)
 customers are enumerated sequentially according to their time of entry into the system (if customers depart the system, they do not return).

Write $q_i (t) \in \N$ for the queue occupied by customer~$i \in \N$
at time $t \in \RP$; we set $q_i (t) =0$ if, by time $t$, the $i$th customer
either did not yet enter to the system, or already departed (that is, all exogenous customers start at state~$0$, and all customers
leaving the system are absorbed at~$0$).
When the customers queue at a server, they may be
served according to some priority policies that we 
do not specify at this point (see Remark~\ref{rem:service}).
Thus we do not determine the holding-time distributions of
the process~$q_i$, but its jump chain performs a random walk.

 Formally,
 denote by $\sigma_i(0)$ the time 
 when the $i$th customer enters into the system
 (in particular, $\sigma_i(0)=0$ for any endogenous
 customers), and let $\sigma_i(n)$
 be the holding time of the $i$th customer
 before it makes its $n$th jump. 

Denote by $Q_i$ the jump chain (or skeleton random walk)
corresponding to $q_i$: 
\[
 Q_i (n)= q_i(\sigma_i(0)+\cdots+\sigma_i(n)), \text{ for } n \in \ZP, 
\]
i.e., the process $Q_i$ is the discrete-time random walk on $\N$ obtained from observing the
sequence of queues visited by customer~$i$.

Consider now a modified queueing system, in which each M/M/1 queue is replaced by an M/M/$\infty$ queue.
That is (cf.~Section~\ref{sec:jackson-terminology}),
the  exogenous 
arrival rates $\lambda$ are given by~\eqref{eq:lambda-def},
service rates $\mu$ are given by~\eqref{eq:mu-def},
and routing matrix $P$ is given by~\eqref{eq:P-def}, but now \emph{all} customers waiting at queue~$k$ are served at the exponential rate~$\mu_k$. In other words, customers encounter service as if they were alone at the queue, and so progress independently through the system. We denote by
$\heta = (\heta_k)_{k \in \N}$ this M/M/$\infty$ system, and by $\hat q_i(t)$ the queue occupied by customer~$i \in \N$ at time $t \in \RP$. Then $\hat q_i$ has the same (in law)
jump chain~$Q_i$ as does $q_i$.
Moreover, if $\hsig_i (n)$ denotes the holding time of the $i$th customer
 before it makes its $n$th jump in the M/M/$\infty$ system,
 we have that the random variables
 $(\hsig_i(n); i,n\in \N)$   are conditionally 
 independent given $(Q_i(n); i,n\in \N)$ 
and such that $\hsig_i(n)$ has exponential
 distribution with rate $\mu_{Q_i(n)}$. 

Then, it is straightforward to see that we have the following.
\begin{proposition}
\phantomsection
\label{prop:customer-random-walk}
\begin{enumerate}[label=(\roman*)]
\item
 \label{prop:customer-random-walk-i}
 The processes 
 $(Q_i)_{i\in \N}$ are independent
 random walks on $\ZP$, each with transition matrix~$P$ defined in~\eqref{eq:P-def} with 
 absorption on first hitting of~$0$.
 \item
 \label{prop:customer-random-walk-ii}
The holding times in the  M/M/1 system and in the  M/M/$\infty$ system
are such that $(\sigma_i(n); i,n\in \N)$ stochastically dominate $(\hsig_i(n); i,n\in \N)$.
\end{enumerate}
\end{proposition}

Proposition~\ref{prop:customer-random-walk}\ref{prop:customer-random-walk-i}
says that all the $Q_i$ perform random walks on $\ZP$ with the same law. It is thus convenient
to introduce the following terminology for a generic such walk.

\begin{definition}[Customer random walk]
\label{def:customer-random-walk}
Denote by $Q = (Q(n))_{n \in \ZP}$
a discrete-time Markov chain on $\ZP$ with transition matrix~$P$ defined in~\eqref{eq:P-def}.
We call $Q$ the~\emph{customer random walk} associated to the queueing network. 
\end{definition}

Assuming~\eqref{ass:positive-rates}, the Markov chain $Q$ is an irreducible nearest-neighbour random walk. The following fact is classical (see e.g.~\cite[\S I.12]{chung} or~\cite[\S 2.2]{mpw-book}). We define the
first return time to the origin by
\begin{equation}
    \label{eq:tau-def}
    \tau := \inf \{ n \in \N : Q(n) = 0\};
\end{equation}
here the usual convention that $\inf \emptyset := +\infty$ is in force.

\begin{lemma}
\label{lem:customer-walk}
Suppose that~\eqref{ass:positive-rates} holds. Then 
the random walk $Q$ is transient if and only if $\sum_{k \in \N} ( a_{k+1} \alpha_k )^{-1} < \infty$
    and positive recurrent if and only if $\sum_{k \in \N}    a_{k+1} \alpha_k   < \infty$. Moreover,
  \begin{equation}
  \label{eq:p0-def}
  p_0 := \IP ( \tau = \infty \mid Q(0) = 1 ) = \left( 1 + \frac{b_1}{a_2} + \frac{b_1 b_2}{a_2 a_3} + \cdots \right)^{-1} \in [0,1) . \end{equation}
\end{lemma}

In particular, as indicated in Remark~\ref{rem:p0-v0},
comparing~\eqref{eq:v0-def}
and~\eqref{eq:p0-def} shows that $|v_0| = a_1 p_0$,
so that $v_0 =0$ whenever $p_0 =0$ ($Q$ recurrent) and
$v_0 < 0$ whenever $p_0 > 0$ ($Q$ transient).

\section{Admissible solutions of the stable traffic equation}
\label{sec:admissible}

The goal of this section is a closer
examination of the set $\cV$
and the associated solutions~$\rho$
of the stable traffic equation~\eqref{eq:stable-traffic},
as defined at Definition~\ref{def:admissible}. The $v \in \cV$ and associated~$\rho = \alpha+v\beta$ govern
invariant measures~$\nu_\rho$ and speeds of dynamics for the particle system,
as presented in Section~\ref{sec:main-results}, so understanding their properties
is of key importance. We will almost always assume the positivity condition~\eqref{ass:positive-rates},
and we will frequently impose the 
non-explosion condition~\eqref{ass:a-inverse-sum}. In a few places we will impose the uniform boundedness condition~\eqref{ass:bounded-rates}. The main focus of this section
is to present and prove Proposition~\ref{prop:admissible-solutions} below (from which Proposition~\ref{prop:admissible-solutions-first} is readily deduced), together with a collection of illustrative examples. First we give a proof of the basic Lemma~\ref{lem:solution-space}.

\begin{proof}[Proof of Lemma~\ref{lem:solution-space}]
First, consider~\eqref{eq:stable-traffic} without the
 constraint on~$\rho_0$.
The solutions~$\rho$ to~\eqref{eq:stable-traffic} form a linear 
subspace of $\R^{\ZP}$, and this subspace is two-dimensional 
because~$\rho_0$ and~$\rho_1$
 uniquely determine the rest. We claim that $\alpha = (\alpha_k)_{k \in \ZP}$
and $\beta = (\beta_k)_{k \in \ZP}$
as defined at~\eqref{eq:alpha-k-def} and~\eqref{eq:beta-k-def}
constitute a basis for the solution space.
The  vectors $\alpha,\beta \in \R^{\ZP}$ are linearly
independent (because one is strictly positive and the other
is not). Moreover, it is straightforward to check that both $\alpha$ and $\beta$
solve the system~\eqref{eq:stable-traffic}: this is familiar from the
usual solution to the difference equations associated with the general gambler's ruin problem~\cite[pp.~106--108]{kt}.
Then, since $\rho_0=1$, 
any solution of~\eqref{eq:stable-traffic} with the given boundary condition
must have the form 
$\rho = \alpha + v \beta$ for some $v \in\R$, which is precisely~\eqref{eq:rho-explicit}. 

Observe from
 the definitions of $\alpha_k$ and $\beta_k$
at~\eqref{eq:alpha-k-def} and~\eqref{eq:beta-k-def}
 that
\begin{equation}
\label{eq:beta-alpha-recursions}
 \alpha_{k+1} = \frac{a_{k+1}}{b_{k+1}} \alpha_k  \text{ and } \beta_{k+1} = \frac{a_{k+1}}{b_{k+1}} \beta_k + \frac{1}{b_{k+1}} , \text{ for all } k \in \ZP.\end{equation}
 We obtain from~\eqref{eq:stable-traffic} and~\eqref{eq:beta-alpha-recursions} that $v = (\rho_k-\alpha_k) / \beta_k = ( \rho_{k+1}-\alpha_{k+1} ) / \beta_{k+1}$. After some algebra, this yields
\begin{equation}
\label{eq:rho_k+1_rho_k}
 \rho_{k+1} = \frac{a_{k+1}}{b_{k+1}} \rho_k + \frac{v}{b_{k+1}}, \text{ for every } k \in \ZP,
\end{equation}
from which~\eqref{eq:v-difference} follows.
\end{proof}

As already observed in Section~\ref{sec:main-results},
 there may be \emph{no} admissible solution to~\eqref{eq:stable-traffic} (i.e., $\cV = \emptyset$), or there may be many (i.e., non-uniqueness). For the former case, see Remark~\ref{rem:non-admissible}. In the latter case,
 as explained in  Section~\ref{sec:main-results}, a distinguished role is played by the \emph{minimal} solution, associated to speed~$v_0$ given at~\eqref{eq:v0-def},
 defined at Definition~\ref{def:minimal-solution}. We next look in more detail at the set~$\cV$ and the role of~$v_0$.

If $v \in \cV$, with a
corresponding
admissible solution $\rho = \alpha + v \beta$, 
we obtain from~\eqref{eq:v-difference} that, since $\rho_k \in [0,1]$ for all $k \in \ZP$,
$- \inf_{k \in \ZP} a_k\leq v \leq \inf_{k \in \ZP} b_k$,
and so
\begin{align}
\label{eq:V-bound}
\cV \subseteq \Bigl[ - \inf_{k\in\ZP} a_k , \inf_{k \in \ZP} b_k \Bigr]. 
\end{align}
Moreover, the fact that $\alpha_k, \beta_k \geq 0$ readily yields  the following two closure properties of $\cV$:
\begin{align}
\label{eq:V-interval}
&{} \text{if $v, v' \in \cV$ with $v < v'$, then $[v,v'] \subseteq \cV$;} \\
\label{eq:V-central}
&{} \text{if $v \in \cV$ with $v > 0$, then $[0,v] \subseteq \cV$.}
\end{align}
The following result describes the permissible
structures of the set $\cV$,  identifies explicitly the 
least member $v_0$,
defined through~\eqref{eq:v0-def},
of (non-empty) $\cV$,
and gives some sufficient conditions for  $\cV$ to contain at most one element (necessarily, $v_0$) and
for $v_0 <0$ or $v_0 =0 $.

\begin{proposition}
\label{prop:admissible-solutions}
Suppose that~\eqref{ass:positive-rates} holds. Then there always exists the limit~$v_0$
defined via~\eqref{eq:v0-def}.
Define
\begin{equation}
    \label{eq:beta-bar-def}
    \abar := \limsup_{k \to \infty} \alpha_k \in [0,\infty], \text{ and }
    \bbar := \limsup_{k \to \infty} \beta_k \in [0,\infty].
\end{equation}
Then the following hold; apply conventions $1/\infty := 0$ and $1/0:=\infty$.
\begin{enumerate}[label=(\roman*)]
 \item
 \label{prop:admissible-solutions-i}
Either $\cV = \emptyset$ (empty);
$\cV = \{ v_0\}$ (a singleton); or else there exists $v_1$ with $0 < v_1-v_0 \leq 1/\bbar$ for which
$\cV = [ v_0, v_1]$ or $\cV = [v_0, v_1 )$ (an interval).
  \item 
   \label{prop:admissible-solutions-ii}
The following implications are valid:
\begin{equation}
\abar = \infty ~\Longrightarrow ~ \bbar = \infty ~ \Longrightarrow ~ \card \cV \leq 1.
    \end{equation}
\item
\label{prop:admissible-solutions-iii}
  The following  equivalences are valid:
  \[ v_0 < 0 ~ \Longleftrightarrow ~ \sum_{k\in \ZP} (a_{k+1} \alpha_k )^{-1} < \infty 
 ~ \Longleftrightarrow ~ \sum_{k\in \N} (b_{k} \alpha_k )^{-1} < \infty. \]
    \item 
    \label{prop:admissible-solutions-iv}
If~\eqref{ass:a-inverse-sum} holds, then either $\card \cV \leq 1$, or else $v_0 = 0$.
 \item 
\label{prop:admissible-solutions-v}
If~\eqref{ass:a-inverse-sum} holds and 
there exists an admissible $\rho$ 
with $\sum_{k \in \ZP} \rho_k < \infty$, 
then $v_0 = 0 \in \cV$ and $\sum_{k \in \ZP} \alpha_k < \infty$.
\item 
\label{prop:admissible-solutions-vi}
If~\eqref{ass:bounded-rates} holds and
there exists an admissible $\rho$ with $\liminf_{k \to \infty} \rho_k =0$,
then $v_0 = 0$ and $\cV = \{0\}$.
 \end{enumerate}
\end{proposition}
 
Note that
Proposition~\ref{prop:admissible-solutions}\ref{prop:admissible-solutions-iii}
combined with Lemma~\ref{lem:customer-walk} 
verifies the characterization of~$v_0$ in terms of the customer
walk~$Q$ given at~\eqref{eq:v0-recurrnce} in Remark~\ref{rem:p0-v0}.

Before the proof of Proposition~\ref{prop:admissible-solutions}, we give some examples. In the first three, 
the $a_i, b_i$ are bounded so that~\eqref{ass:bounded-rates}
holds.
 We show examples with $\cV = \emptyset$ (Example~\ref{ex:homogeneous}),
with $\cV = \{ v_0\}$ for $v_0 =0$ (Example~\ref{ex:lattice-atlas})
and for $v_0 < 0$ (Example~\ref{ex:one-sheep-many-dogs}),
and with~$\cV$ an interval with left endpoint~$v_0$ (with $v_0=0$ in Example~\ref{ex:homogeneous},
and $v_0 <0$ in  Example~\ref{ex:very-fast});
Proposition~\ref{prop:admissible-solutions}\ref{prop:admissible-solutions-i} shows that there are essentially no other possibilities for the form of $\cV$ when~\eqref{ass:bounded-rates} is satisfied.

\begin{example}[Dog and sheep]
\label{ex:lattice-atlas}
Recall the ``dog and sheep'' example described in Example~\ref{ex:lattice-atlas-first},
where~\eqref{eq:dog-sheep-rates} is satisfied. 
Then $\alpha_k = a/b \in (0,1)$, $\beta_k = \frac{k-1}{c} + \frac{1}{b}$, and the customer walk~$Q$
is null recurrent, by Lemma~\ref{lem:customer-walk}. 
Moreover, $\abar < \infty$, $\bbar = \infty$,
$v_0 = 0$, $\cV = \{ 0 \}$, and the unique admissible solution is $\rho = \alpha$.
\end{example}

 \begin{example}[Homogeneous simple exclusion]
\label{ex:homogeneous}
   Recall Examples~\ref{ex:simple-exclusion} and~\ref{ex:homogeneous-first}, where $a_k \equiv a \in (0,\infty)$ and $b_k \equiv b \in (0,\infty)$. 
 Then $\alpha_k = (a/b)^k$ and, if $a \neq b$, $\beta_k = (1 - (a/b)^k)/(b-a)$,
with $\beta_k = k/a$ if $a=b$. 
\begin{itemize}
    \item 
If $a < b$ (so $\abar, \bbar < \infty$), then the customer walk $Q$ is positive recurrent, and $v_0 = 0$.
Solutions $\rho_k = \alpha_k + v \beta_k$ are admissible for $v \in \cV = [ 0, b-a)$.
The minimal solution has $\sum_{k \in \ZP} \alpha_k<\infty$,
while the non-minimal solutions have $\lim_{k \to \infty} \rho_k = \frac{v}{b-a} \in (0,1)$.
\item
If $a \geq b$, then $Q$ is either null recurrent or transient, and one can check that $v_0 = b-a$ but the admissibility condition  fails (since $\alpha_k + v_0 \beta_k \equiv 1$), so that $\cV = \emptyset$. \qedhere
\end{itemize}
\end{example}

\begin{example}[One sheep, many dogs]
\label{ex:one-sheep-many-dogs}
Suppose that $a_1 = b_1 =1$, and $a_k \equiv a$, $b_k \equiv b$ for $k \geq 2$, with $a > b$.
Then $\alpha_k = (a/b)^{k-1}$, $\beta_k = \frac{1-(a/b)^{k-1}}{b-a} + (a/b)^{k-1}$, from which it follows that $v_0 = - ( \frac{a-b}{1+a-b} )$ and $\abar = \bbar = \infty$,
so $\# \cV \leq 1$. The customer walk $Q$ is transient. 
Then $\rho_k = \alpha_k + v_0 \beta_k = \frac{1}{1+a-b} \in (0,1)$ for all $k \in \N$, so this indeed is the unique admissible solution.  Clearly $\sum_{k \in \ZP} \rho_k = \infty$. The corresponding measure $\nu_\rho$ puts no mass on $\bbDB$, so for convergence we are in the setting of Theorem~\ref{thm:local-convergence} (and not Theorem~\ref{thm:finitely-supported-convergence}).
 Theorems~\ref{thm:strong-law} and~\ref{thm:local-convergence} show convergence in the neighbourhood of the leftmost particle, and a strong law with the negative speed~$v_0$.
\end{example}

Here is a (somewhat pathological) example in which there are can be infinitely many $v \in \cV$, of either (or both) signs; necessarily $\abar < \infty$, $\bbar < \infty$, and~\eqref{ass:a-inverse-sum} is violated.

\begin{example}
\label{ex:very-fast}
Suppose that $a_k = a k!$ and $b_k = (k+1)!$, $k \in \N$, where $a \in (0,\infty)$.
Then $a_k/b_k = a/(k+1)$ and $\alpha_k = a^k/(k+1)! = a^k/b_k$,
and so (e.g.~by~\eqref{eq:beta-alpha-frac} below) $\beta_k/\alpha_k = \sum_{j=1}^{k} a^{-j}$.
\begin{itemize}
    \item If $a=1$, then $\beta_k/\alpha_k = k$, then $v_0 = 0$
and $\rho_k = \alpha_k + v \beta_k = (1+ kv )/(k+1)!$ is admissible if and only if $0 \leq v < 1$,
since $k/(k+1)! \leq 1/2$ for $k \in \N$. Thus $\cV = [0,1)$ and for every $v \in \cV$ the corresponding $\rho = \alpha + v \beta$ satisfies $\sum_{k \in \ZP} \rho_k < \infty$. Note that there is no contradiction with Proposition~\ref{prop:admissible-solutions}\ref{prop:admissible-solutions-vi}, since~\eqref{ass:bounded-rates} fails. The customer walk~$Q$ is null recurrent, since $b_k = a_{k+1}$ for all $k \in \N$. 
    \item
    If $a \in (0,1)$, then $\beta_k/\alpha_k = \frac{a^{-k}-1}{1-a}$, so $v_0 = 0$
    and $\rho_k = \alpha_k + v \beta_k = \frac{1}{(k+1)!} ( a^k + v ( \frac{1-a^k}{1-a} ) )$
     is admissible if and only if $0 \leq v < 2-a$, i.e., $\cV = [ 0, 2-a )$.
Since $a_{k+1}/b_k = a <1$, the customer walk~$Q$ is positive recurrent.
     \item If $a>1$, then $\beta_k/\alpha_k = \frac{1-a^{-k}}{a-1}$, so $v_0 = 1-a <0$,
     and solution $\rho_k = \alpha_k + v \beta_k = \frac{1}{(k+1)!} ( a^k + v ( \frac{a^k-1}{a-1} ) )$
     is admissible if and only if $1-a \leq v < 2-a$. Thus $\cV = [1-a,2-a)$. The customer walk~$Q$ is now transient. If $a \in (1,2)$, this gives an example where $\cV$ contains both positive and negative values.
     \end{itemize}
     We do not address here construction of a process with these parameters; one would need to accommodate
     explosion of customers as described in Section~\ref{sec:jackson-terminology} and Remark~\ref{rem:unbounded-rates}.
\end{example}

\begin{proof}[Proof of Proposition~\ref{prop:admissible-solutions}]
It follows from~\eqref{eq:beta-alpha-recursions} that
\begin{equation}
\label{eq:beta-over-alpha-recursion}
 \frac{\beta_{k+1}}{\alpha_{k+1}} = 
 \frac{\beta_{k}}{\alpha_{k}} +  \frac{1}{a_{k+1} \alpha_{k}}  =  \frac{\beta_{k}}{\alpha_{k}}  + \frac{b_1 \cdots b_{k}}{a_1 \cdots a_{k+1}} , \text{ for all } k \in \ZP.\end{equation}
Then, from~\eqref{ass:positive-rates} and~\eqref{eq:beta-over-alpha-recursion}, we see that $\beta_k/\alpha_k$ is strictly increasing, and $\alpha_k/\beta_k$ is strictly decreasing (and non-negative).
Thus $v_0 := - \lim_{k \to \infty} \alpha_k/\beta_k$ exists, and, by monotonicity,
for all $k \in \N$, it holds that 
 $0 \leq | v_0 | < \alpha_k/\beta_k \leq \alpha_1 /\beta_1 = a_1$. Hence $v_0 \in (-a_1, 0]$, as claimed,
and, moreover, $\alpha_k + v_0 \beta_k >0$ for all $k \in \ZP$. 

A consequence of~\eqref{eq:beta-over-alpha-recursion} and the fact that $\beta_0/\alpha_0 = 0$ is that
\begin{equation}
\label{eq:beta-alpha-frac}
\frac{\beta_{k}}{\alpha_{k}} = \sum_{\ell = 0}^{k-1} \left[   \frac{\beta_{\ell+1}}{\alpha_{\ell+1}} -  \frac{\beta_{\ell}}{\alpha_{\ell}} \right] 
= \sum_{\ell = 0}^{k-1} \frac{1}{a_{\ell+1} \alpha_\ell} 
= \sum_{\ell = 0}^{k-1} \frac{1}{b_{\ell+1} \alpha_{\ell+1}} 
. \end{equation}
Hence the following quantities in $[0,+\infty]$ exist (as monotone limits) and are equal:
\[ \lim_{k \to \infty} \frac{\beta_{k}}{\alpha_{k}} = \sum_{k \in \ZP} \frac{1}{a_{k+1} \alpha_k}
= \sum_{k \in \ZP}  \frac{1}{b_{k+1} \alpha_{k+1}}.
\]
By~\eqref{eq:v0-def}, $v_0 = 0$ if and only if the preceding display is infinite.
This proves~\ref{prop:admissible-solutions-iii}.

Suppose that $v \in \cV$ with $v > v_0$. Then $\alpha_k + v_0 \beta_k < \alpha_k + v \beta_k < 1$ for all $k \in \ZP$, since $v \in \cV$, and so
\begin{equation}
    \label{eq:v0-minimal}
    \text{if $v \in \cV$ with $v > v_0$, then $v_0 \in \cV$.}
\end{equation}
Define $v_\star := \inf \cV$, which, by~\eqref{eq:V-central} satisfies $v_\star \leq 0$ when $\cV \neq \emptyset$.
Suppose that $\cV \neq \emptyset$; we will show that $v_\star = v_0$.
Let $\rho = \alpha + v \beta$, $v \in \cV$, be any admissible solution to~\eqref{eq:stable-traffic};
then $\alpha_k + v \beta_k >0$ for all $k \in \ZP$,
i.e., $\alpha_k / \beta_k > -v$ for all $k \in \ZP$. 
In particular,
$- v_0 = \liminf_{k \to \infty} ( \alpha_k / \beta_k ) \geq -v $ for every $v \in \cV$, and hence $v_0 \leq v_\star$.
Moreover,
by definition of $v_\star$, there exists $v > v_\star \geq v_0$ with $v \in \cV$, and hence
from~\eqref{eq:v0-minimal} we conclude that $v_0 \in \cV$. 
Thus $v_\star \leq v_0$, and so in fact $v_0 = v_\star$.
We have thus established that
\begin{equation}
    \label{eq:v0-minimal-2}
    \text{if and only if $\cV \neq \emptyset$, it holds that $v_0 = \inf \cV$ and $v_0 \in \cV$}.
\end{equation}
Suppose that $v_0 \in \cV$, and $v > v_0$. Then
$v \in \cV$ if and only if 
$v < (1-\alpha_k)/\beta_k$.
Let
\[ v_2 := \inf_{k \in \N} \frac{1-\alpha_k}{\beta_k} 
\leq \liminf_{k \to \infty} \frac{1}{\beta_k} - \lim_{k \to \infty } \frac{\alpha_k}{\beta_k}
= v_0 + \bigl( 1 / \bbar \bigr)  ,
\]
where $\bbar$ is given at~\eqref{eq:beta-bar-def}. 
If $\bbar = \infty$, then $v_2 = v_0$ and so $\cV \subseteq \{ v_0\}$.  Moreover,
 since $a_1 \beta_k > \alpha_k$ for all $k \in \ZP$, it is immediate that $\bbar$ is infinite whenever $\abar$ is infinite. This verifies the  sequence of implications in~\ref{prop:admissible-solutions-ii}.
In addition, by~\eqref{eq:V-interval} and~\eqref{eq:v0-minimal-2}, either $\cV = \emptyset$,
$\cV = \{ v_0\}$, or $\cV$ is an interval closed at least at the left endpoint ($v_0$) and with right endpoint $v_1$ with $v_0 < v_1 \leq v_0 + (1/\bbar)$; this proves part~\ref{prop:admissible-solutions-i}. 

We next prove part~\ref{prop:admissible-solutions-iv}.
To do so, note that there exists $k_0 \in \N$ such that $\alpha_k \leq 2 \abar$ for all $k \geq k_0$,
and so $\sum_{k \in \ZP} ( a_{k+1} \alpha_k )^{-1} \geq (2 \abar)^{-1} \sum_{k \geq k_0} a_{k+1}^{-1}$
whenever $\abar < \infty$. 
From part~\ref{prop:admissible-solutions-iii} (proved earlier) we then observe that
\begin{align}
\label{eq:a-inverse-sum-consequence}
\text{if~\eqref{ass:a-inverse-sum} holds, then either $v_0 = 0$ or $\abar = \infty$.}
\end{align}
Combining~\eqref{eq:a-inverse-sum-consequence} 
with part~\ref{prop:admissible-solutions-ii}, we obtain part~\ref{prop:admissible-solutions-iv}.

Next, suppose there exists an admissible solution $\rho = \alpha + v \beta$ for which $\sum_{k \in \ZP} \rho_k < \infty$. 
If~$\rho$ is admissible, 
 it follows from~\eqref{eq:v-difference} and the fact that $\rho_k >0$ for all $k$, that
 \begin{equation}
 \label{eq:v-between-limits}
  -a_{k+1}\rho_k < v < b_k\rho_k, \text{ for every } k \in \ZP.
 \end{equation}
 Since $v_0 \in \cV$ and $v_0 \leq 0$, we have from~\eqref{eq:v-between-limits} that $a_{k+1} \rho_k > | v_0 |$. If $v_0 < 0$, this means that $\sum_{k \in \ZP} \rho_k < \infty$
implies $\sum_{k \in \ZP} (1/a_k) < \infty$, which contradicts~\eqref{ass:a-inverse-sum}. Thus $v_0 = 0$.
Hence for $v \in \cV$, we have $v \geq 0$ and $0 < \alpha_k = \rho_k - v \beta_k \leq \rho_k$, meaning that
$\sum_{k \in \ZP} \rho_k < \infty$ implies that $\sum_{k \in \ZP} \alpha_k < \infty$ as well. 
This establishes part~\ref{prop:admissible-solutions-v}.
Finally, suppose that~\eqref{ass:bounded-rates} holds and 
there exists an admissible solution $\rho = \alpha + v \beta$ for which $\liminf_{k \to \infty} \rho_k = 0$. 
Then it follows from~\eqref{eq:v-between-limits} that $v =0$ for every $v \in \cV$.
\end{proof}
 
Finally, we give the proof of
Proposition~\ref{prop:admissible-solutions-first}.

\begin{proof}[Proof of Proposition~\ref{prop:admissible-solutions-first}]
Existence of $v_0$ given by~\eqref{eq:v0-def}
is established in Proposition~\ref{prop:admissible-solutions},
as is the fact that whenever $\cV \neq \emptyset$, it holds that
$\inf \cV = v_0 \in \cV$. Moreover,
from~\eqref{eq:beta-over-alpha-recursion} and~\eqref{ass:positive-rates}
we have that $\alpha_k / \beta_k$ is strictly decreasing, so that $0 < \alpha_k + v_0 \beta_k$
for all $k \in \ZP$ holds automatically. Thus $v_0 \in \cV$ if and only if $\alpha_k + v_0 \beta_k < 1$ for all $k$, i.e., $v_0 < (1-\alpha_k)/\beta_k$ for all $k$, which, by~\eqref{eq:alpha-k-def} and~\eqref{eq:beta-k-def}, is equivalent to~\eqref{eq:v0-admissible}.
\end{proof}

\section{Invariant measures, domination,  and convergence}
\label{sec:invariant-measures}

\subsection{Invariant measures and finitely-supported convergence}
 
Recall the definition of the product-geometric measure $\nu_\rho$ on $\bbD = \ZP^\N$
from~\eqref{eq:df_measure_rho},
and that, by Proposition~\ref{prop:invariant-measure}, the~$\nu_\rho$, for admissible~$\rho$,
serve as stationary measures.
The main subject of this section is to demonstrate that,
 under appropriate conditions,
there is convergence to the~$\nu_\rho$ associated with the minimal admissible~$\rho$;
we give proofs of 
Theorems~\ref{thm:finitely-supported-convergence} (in this section) and~\ref{thm:local-convergence} (in Section~\ref{sec:convergence} below).

Recall the definition of $\bbDB$ from~\eqref{eq:D-F-def}.
For an admissible~$\rho$, we say that the corresponding measure~$\nu_\rho$
    defined by~\eqref{eq:df_measure_rho} is \emph{supported on finite configurations} if~$\nu_\rho ( \bbDB ) = 1$.
    The next lemma gives a simple dichotomy based on the following condition:
\begin{equation}
\label{eq:rho_summable}
 \sum_{k \in \ZP} \rho_k < \infty.
\end{equation}
Write $\IE_{\nu_\rho}$ for expectation corresponding to the probability measure~$\nu_\rho$.

\begin{lemma}
\label{lem:stationary-mean}
    Suppose that $\rho$ is admissible. If~\eqref{eq:rho_summable} holds, then $\IE_{\nu_\rho}  \sum_{k \in \N} \eta_k < \infty$,
    and hence $\nu_\rho ( \bbDB ) = 1$;
    otherwise it holds that    $\nu_\rho ( \bbDB ) = 0$.
\end{lemma}
\begin{proof}
The distribution $\geo{q}$ has mean $\frac{1-q}{q}$ (recall it lives on $\ZP$), so the component $\geo{1-\rho_k}$ of $\nu_\rho$ given by~\eqref{eq:df_measure_rho} has mean $\frac{\rho_k}{1-\rho_k}$. Thus
\begin{equation}
\label{eq:stationary-mean}
\IE_{\nu_\rho}  \sum_{k \in \N} \eta_k 
= \sum_{k \in \N} \frac{\rho_k}{1-\rho_k}. 
\end{equation}
Note that~\eqref{eq:rho_summable}  implies $\lim_{k \to \infty} \rho_k =0$. 
    If $\rho$ is admissible, so that~$\rho_k \in (0,1)$
for all $k \in \N$,
then $\sum_{k \in \N} \frac{\rho_k}{1-\rho_k} < \infty$
if and only if~\eqref{eq:rho_summable} holds, and so~\eqref{eq:stationary-mean} yields the first half of the statement in the lemma. 
Moreover, $ {\nu_\rho} \{ \eta \in \bbD : \eta_k \geq 1 \} = \rho_k$,
and, under $\nu_\rho$, the $\eta_k$ are independent; the Borel--Cantelli lemma then 
verifies the second half of the statement.
\end{proof}

Define
\begin{equation}
    \label{eq:VF}
    \cVF := \Bigl\{ v \in \cV : \sum_{k \in \ZP} ( \alpha_k + v \beta_k ) < \infty \Bigr\}. 
\end{equation}
In view of Lemma~\ref{lem:stationary-mean},
  the $\rho = \alpha + v \beta$ with $v \in \cVF$ are the admissible solutions corresponding to \emph{finitely supported} measures. The following observations are straightforward.

\begin{lemma}
\label{lem:finitely-supported}
Suppose that~\eqref{ass:positive-rates} holds. 
Exactly one of the following holds.
\begin{enumerate}[label=(\roman*)]
 \item
 \label{lem:finitely-supported-i}
  $\sum_{k \in \ZP} \alpha _k < \infty$ and $\sum_{k \in \ZP} \beta_k < \infty$, and $\cVF = \cV$.
    \item
     \label{lem:finitely-supported-ii}
$\sum_{k \in \ZP} \alpha_k < \infty$
    and $\sum_{k \in \ZP} \beta_k = \infty$, and $\cVF = \cV \cap \{ 0 \}$.
    \item
     \label{lem:finitely-supported-iii}
$\sum_{k \in \ZP} \alpha_k  = \sum_{k \in \ZP} \beta_k = \infty$, and either $\cVF = \emptyset$, or $\cVF = \{ v \}$ for some $v_0 \leq v < 0$.
\end{enumerate}
Moreover, if~\eqref{ass:a-inverse-sum} also holds, then $\sum_{k \in \ZP} \alpha_k < \infty$
implies $\# \cVF \leq 1$, while  $\sum_{k \in \ZP} \alpha_k = \infty$
implies $\cVF = \emptyset$.
\end{lemma}
\begin{proof}
    First note that, since $\beta_k \geq \alpha_k/a_1$,
    if $\sum_{k \in \ZP} \alpha_k = \infty$
    then $\sum_{k \in \ZP} \beta_k = \infty$ as well. Recalling that $0 \geq v_0 = \inf \cV$,
    it is elementary to establish~\ref{lem:finitely-supported-i}--\ref{lem:finitely-supported-iii}.

Suppose, additionally, that~\eqref{ass:a-inverse-sum} holds.
 Then Proposition~\ref{prop:admissible-solutions}\ref{prop:admissible-solutions-v} shows that in
    case~\ref{lem:finitely-supported-iii}, it must be the case that $\cVF = \emptyset$. On the other hand,
    if  $\sum_{k \in \ZP} \alpha_k < \infty$ then $\nu_\alpha$
    is an invariant measure for the irreducible, countable Markov chain~$\eta$ on $\bbDB$,
    and hence $\eta$ is positive recurrent and the invariant measure is unique, by e.g.~Theorem~3.5.3
of~\cite[p.~118]{norris}.
\end{proof}

If~\eqref{ass:bounded-rates} holds,  then Proposition~\ref{prop:admissible-solutions}\ref{prop:admissible-solutions-vi}
says that $\cVF = \cV = \{ 0 \}$ whenever $\cVF$ is non-empty. 
We can now complete
the proofs of Theorem~\ref{thm:finitely-supported-convergence}
and Corollary~\ref{cor:finitely-supported-convergence}.

\begin{proof}[Proof of Theorem~\ref{thm:finitely-supported-convergence}]
Suppose that~\eqref{ass:positive-rates} and~\eqref{ass:a-inverse-sum} hold, that
$\sum_{k \in \ZP} \alpha_k  < \infty$, and $X(0) \in \bbXB$.
By Proposition~\ref{prop:existence-from-heaviside},
the process~$\eta$
is an irreducible Markov chain on the countable state space $\bbDB$,
and, by Lemma~\ref{lem:finitely-supported}, 
the (minimal) admissible solution~$\rho = \alpha$  corresponds to
a unique invariant probability measure~$\nu_\alpha$ for $\eta$. 
Hence $\eta$ is positive recurrent
and irreducible, and convergence follows from the convergence theorem
for continuous-time, countable state-space Markov chains
(e.g.~Theorem~3.6.2
of~\cite[p.~122]{norris}).
\end{proof}

\begin{proof}[Proof of Corollary~\ref{cor:finitely-supported-convergence}]
Under the hypotheses of Corollary~\ref{cor:finitely-supported-convergence},
we have that $X(t) \in \bbXB$ for all $t \in \RP$,
and so, by~\eqref{eq:conservation}, we have that
 $X_1 (t) - X_1 (0) = - \sum_{k\in\N} (\eta_k (t) - \eta_k (0))$.
 Since $\eta$ is a positive-recurrent Markov chain
 and the unique invariant measure $\nu_\rho$
 is such that $\nu_\rho ( \bbDB ) = 1$, 
 the ergodic theorem for countable Markov chains in continuous time
 (e.g.~Theorem~3.8.1
of~\cite[p.~126]{norris})
 completes the proof.
 \end{proof}

\subsection{Stochastic domination and second-class customers}
\label{sec:stochastic-domination}

We now turn towards the situation where the minimal admissible solution~$\rho$ 
has $\sum_{k \in \ZP} \rho_k = \infty$,
so that,
 although the process 
 may start from $\eta(0) \in \bbDB$,
Lemma~\ref{lem:stationary-mean} shows that the candidate measure $\nu_\rho$ for convergence is supported on $\bbD \setminus \bbDB$,
and so Theorem~\ref{thm:finitely-supported-convergence} (based on convergence of
countable Markov chains) is not applicable. Here we need some additional arguments,
involving some
stochastic monotonicity, and the concept of \emph{second-class customers}, which will be used in the proof of Theorem~\ref{thm:local-convergence} and elsewhere.

Write $\cP ( \bbD )$ for the set of probability measures on $\bbD$.
Recall the definition of $\nu_\rho \in \cP (\bbD)$ from~\eqref{eq:df_measure_rho}. 
Let 
\[ \cP_\cV ( \bbD) := \{ \nu_\rho : \rho = \alpha + v \beta, v \in \cV \}  ,\]
the set of product-geometric stationary measures on $\bbD$
corresponding to~$\cV$. 
In this section we assume hypothesis~\eqref{ass:bounded-rates};
recall that under this condition, we have (see Proposition~\ref{prop:invariant-measure}) constructed the process~$\eta$ to start from an arbitrary initial state
$\eta (0) \in \bbD$.
For two probability measures $\nu, \nu'$ on $\bbD$, we write
$\nu \stochdoml \nu'$ (equivalently, $\nu' \stochdomg \nu$)
to mean that $\nu$
is stochastically dominated by $\nu'$, i.e.,
$\nu_k [r, \infty) \leq \nu'_k [r,\infty)$
for all $k \in \N$ and all $r \in \ZP$.

\begin{proposition}[Stochastic monotonicity]
\label{prop:stochastic-monotonicity}
Suppose that~\eqref{ass:bounded-rates} holds.
Let $\nu_1, \nu_2 \in \cP ( \bbD)$, and suppose that $\nu_1 \stochdoml \nu_2$.
Then the law of the process $\eta$ started from $\eta(0) \sim \nu_1$
is stochastically dominated by the law of the process $\eta$ started from $\eta(0) \sim \nu_2$.
\end{proposition}

The preceding proposition together with Proposition~\ref{prop:invariant-measure}
(invariant measures) yields the following immediate corollary.

\begin{corollary}
\label{cor:domination}
Suppose that~\eqref{ass:bounded-rates} holds,
 that~$\nu_\rho \in \cP_\cV (\bbD)$, and that $\eta(0) \sim \nu \in  \cP ( \bbD)$.
\begin{itemize}
\item If $\nu \stochdoml \nu_\rho$, then
the law of the process $\eta$ started from $\eta(0) \sim \nu$
is stochastically dominated by the law of the (stationary) process $\eta$ started from $\eta(0) \sim \nu_\rho$.
\item If $\nu \stochdomg \nu_\rho$, then 
the law of the process $\eta$ started from $\eta(0) \sim \nu$
 stochastically dominates the law of the (stationary) process $\eta$ started from $\eta(0) \sim \nu_\rho$.
\end{itemize}
\end{corollary}

The proof of  Proposition~\ref{prop:stochastic-monotonicity} uses a coupling construction,
generally referred to as the \emph{basic coupling}, which 
was observed early on in the theory of interacting particle systems; in
the zero-range context, the idea can be found already in~\cite[\S 2]{andjel82}. 

The version of the coupling construction we use here is most naturally framed in the queueing framework, through the concept of \emph{second-class customers}.
We divide the customers
into two classes, the first and the second; when queueing,
the first-class customers always have priority
over the second-class ones (i.e., no second-class customer
can be served as long as there are first-class customers
in the queue). The (exogeneous) customers that enter the system
at positive times are always first-class.

\begin{proof}[Proof of Proposition~\ref{prop:stochastic-monotonicity}]
Suppose that  $\nu_1 \stochdoml \nu_2$. Then we construct on a suitable probability space a
random element $(\eta(0), \eta'(0)) \in \bbD \times \bbD$ such that $\eta(0) \sim \nu_1$,
$\eta'(0) \sim \nu_2$, and $\eta(0) \leq \eta'(0)$ pointwise, i.e., $\eta_k(0) \leq \eta'_k (0)$
for all $k \in \N$. Declare all of the (endogenous) customers in $\eta(0)$ first class,
and all of those in $\eta'(0) - \eta(0)$ second class.
Let $\eta$ be the process that tracks first-class customers only, and $\eta'$ the process that tracks both classes of customers. The dynamics of first-class customers is oblivious to the second-class customers,
and the $\eta$ system only has first-class customers, so $\eta$
is an honest copy of the process started from $\eta(0)$.
On the other hand, ignoring customer classes, $\eta'$
is an honest copy of the process started from $\eta'(0)$.
By construction, $\eta(t) \leq \eta'(t)$ for all $t \in \RP$.
\end{proof}

\subsection{Comparison with finite systems}
\label{sec:finite-systems}

We use a truncation
idea to couple the semi-infinite
particle system~$\xi$ with large finite systems
of the kind studied in~\cite{mmpw}.
First, for convenience,
we collect the results from~\cite{mmpw} that we need here.
The finite system of $N+1$ particles (hence $N$ queues)
has parameters $a_1, b_1, \ldots, a_{N+1}, b_{N+1}$,
having the same meaning as in Section~\ref{sec:model},
so that particle~$k$ ($1 \leq k \leq N+1$)
performs a nearest neighbour random walk, rate~$a_k$ to the left and $b_k$ to the right,
subject to the exclusion rule. The Jackson network correspondence is similar to the
described in Section~\ref{sec:jackson-terminology} (see~\cite[\S 3]{mmpw}), but now
customers may also enter or depart the system from queue~$N$ (at rates $b_{N+1}$ and $a_{N+1}$, respectively). We denote by $X_{N,k} (t)$ the location of particle $k$ at time $t$, and by $\eta_{N,k} (t) := X_{N,k+1} (t) - X_{N,k} (t) - 1$ the occupancy of queue~$k$.

Recall the definitions of $\alpha_k$ and $\beta_k$ from~\eqref{eq:alpha-k-def} and~\eqref{eq:beta-k-def}. The following result summarizes the relevant results from~\cite{mmpw},
giving finite-$N$ analogues of the strong law (Theorem~\ref{thm:strong-law}) and convergence result (Theorem~\ref{thm:finitely-supported-convergence}).

\begin{proposition}[Finite stable systems; Corollary~2.6 in~\cite{mmpw}]
\label{prop:finite-cloud}
    Let $N \in \N$ and suppose that $a_k \in (0,\infty)$, $b_k \in \RP$ for $1 \leq k \leq N+1$.
   For $1 \leq k \leq N$, set $\rho_{N,k} := \alpha_k + v_{N} \beta_k$, 
   and
    \begin{equation}
\label{eq:v_explicit_variant} 
 v_{N} := \frac{1-\alpha_{N+1}}{\beta_{N+1}} = 
 \frac{b_1 \cdots b_{N+1}- a_1\cdots a_{N+1}}
 {b_1\cdots b_N 
 +b_1 \cdots b_{N-1} a_{N+1} +
  \cdots + a_2 \cdots a_{N+1}}.
\end{equation}
Then if $\rho_{N,k} < 1$ for all $1 \leq k \leq N$, for every $1 \leq k \leq N+1$, we have 
\[ \lim_{t \to \infty} \frac{X_{N,k} (t)}{t} = v_{N}, \as, \]
 and $\eta_N (t)$ converges in distribution as $t \to \infty$
to $\bigotimes_{1 \leq k \leq N} \geo{ 1-\rho_{N,k}}$.
\end{proposition}

Since $\alpha_k/\beta_k$ is
strictly decreasing (see~\eqref{eq:beta-over-alpha-recursion}), $v_N$ defined by~\eqref{eq:v_explicit_variant} satisfies
$v_N > \frac{1}{\beta_{N+1}} - \frac{\alpha_k}{\beta_k}$ for $1 \leq k \leq N$, and hence $\rho_{N,k} > \beta_k / \beta_{N+1}$
for all $1 \leq k \leq N$; in particular, $\rho_{N,k} > 0$ always.
The following example, from~\cite{mmpw}, is the finite-particle
analogue of Example~\ref{ex:lattice-atlas-first}.

\begin{example}[1 dog and $N$ sheep]
\label{ex_N_sheep}
Assume that $a_1=a\in (0,1)$,
and $b_1=a_2=b_2=\cdots=b_{N+1}=1$. 
We know from  Example~2.15 of~\cite{mmpw} 
that $\rho_k =   a + (1-a)\frac{k}{N+1}<1$, 
for $k=1,\ldots,N$,
and the speed of the cloud is
$v_N = \frac{1-a}{N+1}$.
\end{example}

Here is our comparison result.

\begin{proposition}
    \label{prop:finite-coupling}
    Suppose that~\eqref{ass:a-inverse-sum} holds.
\begin{enumerate}[label=(\roman*)]
 \item\label{prop:finite-coupling-i}
Consider the finite particle system $\xi^{N,0}(t)$ 
  with parameters $a_1, b_1$, $\ldots, a_N, b_N, a_{N+1}$
but with $b_{N+1}$ set to~$0$. There exists a coupling of $\xi^{N,0}(t)$
and $\xi$ with $\xi^{N,0} (0) = \xi (0) \in \Z \times \bbDB$ such that,
for all  $k \in \{1,2,\ldots, N\}$ and all $t \in \RP$,
\[ \eta_k^{N,0} (t) \leq \eta_k (t), \text{ and } X_k^{N,0} (t) \leq X_k (t) .\]
 \item\label{prop:finite-coupling-ii}
    Consider the finite particle system $\xi^{N,1}(t)$ 
  with parameters $a_1, b_1, \ldots, a_{N+1} , b_{N+1}$. 
  There exists a coupling of $\xi^{N,1}(t)$
and $\xi$ with $\xi^{N,1} (0) = \xi (0) \in \Z \times \bbDB$ such that,
for all  $k \in \{1,2,\ldots, N\}$ and all $t \in \RP$,
\[ \eta_k^{N,1} (t) \geq \eta_k (t), \text{ and } X_k^{N,1} (t) \geq X_k (t) .\]
\end{enumerate}
    \end{proposition}
\begin{proof}
    In both cases, we will start the construction by building the underlying Jackson networks.
Recall that in the construction of the semi-infinite
Jackson system, as described in Section~\ref{sec:jackson-terminology}, 
we considered the counting processes $E^{\rightarrow 1}(t)$ and 
$E^{\leftarrow 1}(t)$,
the number of  customers that enter and depart, respectively, the system during $(0,t]$.
 Recall from~\eqref{eq:X1-via-counting} that
$X_1 (t) =  X_1 (0) + E^{\leftarrow 1}(t) - E^{\rightarrow 1}(t)$,  for all $t \in \RP$. 

Fix $N \in \N$.  
Let $\eta^{N,0}$ denote the
finite system of $N$ queues with parameters $a_1, b_1, \ldots, a_N, b_N, a_{N+1} \in (0,\infty)$
exactly as in the semi-infinite system $\eta$,
but with $b_{N+1}$ set to~$0$; that is, queue $N+1$ never routes customers to queue~$N$. Then $\eta^{N,0}$
evolves as an autonomous finite open Jackson network, with customer entry via queue~$1$ 
and exit via queues~$1$ and $N$. 
The natural coupling of $\eta^{N,0}$ and $\eta$ 
demonstrates that $\eta^{N,0} \stochdoml \eta$, since
customers in $\eta^{N,0}$ that depart queue~$N$ never return, while those in $\eta$ may at some point come back.
We use the same process 
of exogenous customer arrivals 
as in the semi-infinite system, $E^{\rightarrow 1}(t)$,  the  number of distinct exogenous customers that enter via queue~1  during $(0,t]$. Also introduce notation:
 \begin{itemize}
     \item $E^{\leftarrow 1}_{N,0}(t)$, the number of  customers that depart the system via queue~1 before reaching queue $N+1$ during $(0,t]$.
     \item $E^{N \rightarrow}_{N,0} (t)$, the number of distinct customers that,
     during $(0,t]$, complete at least one service at queue $N$ and are routed to queue~$N+1$.
 \end{itemize}
 The coupling is such that $E^{\leftarrow 1}_{N,0}(t) \leq E^{\leftarrow 1} (t)$,
 since customers that leave the (finite) system via queue~$N$ are lost
 before they can exit via queue~$1$. 
 In the particle system, particle~$N+1$ never jumps to the right, and holes can exit the system on the right but never return.
 Since the leftmost particle only moves on flux through the \emph{left} boundary of the system, 
\[ X_1^{N,0} (t) = X_1 (0)  +E^{\leftarrow 1}_{N,0}(t) - E^{\rightarrow 1}(t)  \leq X_1 (t)  , \text{ for all } t \in \RP, \]
 The gives a coupling of $\xi(t)$
 with the finite particle system $\xi^{N,0}(t)$ 
  with parameters $a_1, b_1, \ldots, a_N, b_N, a_{N+1} \in (0,\infty)$
but with $b_{N+1}$ set to~$0$, and verifies part~\ref{prop:finite-coupling-i}.

Let $\eta^{N,1}$ denote the
finite system of $N$ queues with parameters $a_1, b_1, \ldots,   a_{N+1} , b_{N+1} \in (0,\infty)$.
Then $\eta^{N,1}$
evolves as an autonomous finite open Jackson network, with customer entry and exit via queues~$1$ and $N$. The natural coupling of $\eta^{N,1}$ and $\eta$ 
demonstrates that $\eta^{N,1} \stochdomg \eta$, since 
in $\eta^{N,1}$ exogenous customers enter queue~$N$ at rate $b_{N+1}$,
while in $\eta$ this entry stream is thinned. 
The ``extra'' 
customers that arrive in $\eta^{N,1}$ via queue~$N$
we view as second-class customers in the infinite system, and then ignoring all second-class customers recovers~$\eta$.
Furthermore, introduce notation:
 \begin{itemize}
  \item $E^{\leftarrow 1}_{N,1}(t)$, the number of  customers that depart via queue~1 before reaching queue $N+1$ during $(0,t]$,
  plus the number of second-class customers that enter from the right and depart from the left  during $(0,t]$.
 \end{itemize}
Now the coupling is such that $E^{\leftarrow 1}_{N,1}(t) \geq E^{\leftarrow 1}(t)$, and hence
\[ X_1^{N,1} (t) = X_1 (0)  +E^{\leftarrow 1}_{N,1}(t) - E^{\rightarrow 1}(t)  \geq X_1 (t)  , \text{ for all } t \in \RP, \]
 The gives a coupling of $\xi(t)$
 with the finite particle system $\xi^{N,1}(t)$ 
  with parameters $a_1, b_1$, $\ldots, a_N, b_N, a_{N+1},b_{N+1} \in (0,\infty)$, and verifies part~\ref{prop:finite-coupling-ii}.
 \end{proof}

\subsection{Proof of local convergence}
\label{sec:convergence}

  \begin{proof}[Proof of Theorem~\ref{thm:local-convergence}]
We use the truncation constructions described in Proposition~\ref{prop:finite-coupling}.
Take $A \subset \N$ finite, and choose $N > \sup A$.
Proposition~\ref{prop:finite-coupling}\ref{prop:finite-coupling-i}
gives existence of the finite system $\eta^{N,0}$ for which $\eta^{N,0} \stochdoml \eta$.
We will apply Proposition~\ref{prop:finite-cloud} to $\eta^{N,0}$;
to do so, write $\rho_{N,k} = \alpha_k + v_{N,0} \beta_k$,
where $v_{N,0}$ is given by
the formula for $v_N$ at~\eqref{eq:v_explicit_variant}
but with  $b_{N+1}$ set to zero, i.e., $v_{N,0} = - \alpha_{N+1} / \beta_{N+1}$.
Then, by~\eqref{eq:v0-def} and the fact that $\alpha_k/\beta_k$ is strictly decreasing (see~\eqref{eq:beta-over-alpha-recursion}) it follows that
\begin{equation}
    \label{eq:finite-infinite-coupling}
v_{N,0} \leq \lim_{N \to \infty} v_{N,0} = v_0, \text{ and } \rho_{N,k} \leq \lim_{N \to \infty} \rho_{N,k} = \alpha_k + v_0 \beta_k = \rho_k, \end{equation}
where $\rho$ is the minimal admissible solution.
Note that, since $\rho$ is admissible, we get from~\eqref{eq:finite-infinite-coupling}
that $\rho_{N,k} < 1$ for all $N$ and all $1 \leq k \leq N$. Thus Proposition~\ref{prop:finite-cloud}
is applicable; it yields
\begin{equation}
\label{eq:finite-convergence}
\lim_{t \to \infty} \IP \bigg( \bigcap_{k \in A} \bigl\{ \eta_k^{N,0} (t) \geq m_k \bigr\} \bigg) = \prod_{k \in A} \rho_{N,k}^{m_k} , \text{ for } m = (m_k)_{k \in A} \in \ZP^A.
\end{equation}
Then, from~\eqref{eq:finite-convergence}, the fact that  $\eta^{N,0} \stochdoml \eta$, 
and taking $N \to \infty$ and using~\eqref{eq:finite-infinite-coupling}, we obtain
\begin{equation}
\label{eq:local-convergence-liminf}    
 \liminf_{t \to \infty} \IP \bigg( \bigcap_{k \in A} \bigl\{ \eta_k (t) \geq m_k \bigr\} \bigg) \geq \lim_{N \to \infty} \prod_{k \in A} \rho_{N,k}^{m_k}  = \prod_{k \in A} \rho_{k}^{m_k}.
\end{equation}
On the other hand,
started from $\eta(0) = 0 \in \bbDB$,
Corollary~\ref{cor:domination}
implies that $\eta (t) \stochdoml \nu_\rho$ for all~$t$, so that, in particular,
\begin{equation}
\label{eq:local-convergence-0}   
\IP \bigg( \bigcap_{k \in A} \bigl\{ \eta_k (t) \geq m_k \bigr\} \bigg) 
\leq \prod_{k \in A} \rho_{k}^{m_k}.
\end{equation}

In the following, we use the notation $\IP_\zeta$ for the process started from a fixed
configuration $\eta(0) = \zeta\in \bbD$ and $\IP_\nu$ for the process started from $\eta(0) \sim \nu$ for some measure~$\nu \in \cP ( \bbD)$. First, the combination of~\eqref{eq:local-convergence-liminf} and~\eqref{eq:local-convergence-0} establishes that
\begin{equation}
\label{eq:conv_0_nu_rho}
 \lim_ {t\to\infty}\IP_0 \bigg( \bigcap_{k \in A} \bigl\{ \eta_k (t) \geq m_k \bigr\} \bigg) 
= \prod_{k \in A} \rho_{k}^{m_k}.
\end{equation}
If $\nu\stochdoml \nu_\rho$, we can write, using Proposition~\ref{prop:stochastic-monotonicity}, 
\begin{align*}
\IP_0 \bigg( \bigcap_{k \in A} \bigl\{ \eta_k (t) \geq m_k \bigr\} \bigg) 
& \leq \IP_\nu \bigg( \bigcap_{k \in A} \bigl\{ \eta_k (t) \geq m_k \bigr\} \bigg) \\
& \leq \IP_{\nu_\rho} \bigg( \bigcap_{k \in A} \bigl\{ \eta_k (t) \geq m_k \bigr\} \bigg) 
= \prod_{k \in A} \rho_{k}^{m_k},
\end{align*}
which, together with~\eqref{eq:conv_0_nu_rho},
implies the claim for the process starting from~$\nu$.

Now, let $\eta^*=\eta (0) \in \bbDB$ be a fixed finite initial configuration,
and let~$\zeta$ be a random (initial) configuration chosen according to~$\nu_\rho$.
Abbreviate $u:=\prod_{k \in A} \rho_{k}^{m_k}$ and $p:=\IP_{\nu_\rho} (  \zeta\geq \eta^* ) > 0$,
where the inequality $\zeta\geq \eta^*$ means $\zeta_k \geq \eta^*_k$ for every $k \in \N$.
Fix an arbitrary~$\eps \in (0,u)$. By~\eqref{eq:conv_0_nu_rho}, there exists $t_0 \in \RP$ large enough so that
\[
\IP_0 \bigg( \bigcap_{k \in A} 
\bigl\{ \eta_k (t) \geq m_k \bigr\} \bigg) \geq u-\eps, \text{ for all } t \geq t_0.
\]
Then, we can write
\begin{align*}
u & = \IE_{\nu_\rho} \IP_\zeta \bigg( \bigcap_{k \in A} 
\bigl\{ \eta_k (t) \geq m_k \bigr\} \bigg) \\
&= \IE_{\nu_\rho} \bigg(\big(\1{\zeta\geq \eta^*} + \1{\zeta \not \geq \eta^*}\big)
\IP_\zeta \bigg( \bigcap_{k \in A} 
\bigl\{ \eta_k (t) \geq m_k \bigr\} \bigg)\bigg)\\
&\geq p \cdot \IP_{\eta^*} \bigg( \bigcap_{k \in A} 
\bigl\{ \eta_k (t) \geq m_k \bigr\} \bigg) + (1-p)(u-\eps),
\end{align*}
by two further applications of Proposition~\ref{prop:stochastic-monotonicity},
once comparing $\zeta$ with $\eta^*$ and once $\zeta$ with $0$. 
It follows that
\begin{equation}
\label{eq:local-convergence-1}    
\IP_{\eta^*} \bigg( \bigcap_{k \in A} 
\bigl\{ \eta_k (t) \geq m_k \bigr\} \bigg)
  \leq \frac{u-(1-p)(u-\eps)}{p} = u + \frac{1-p}{p}\eps.
\end{equation}
Moreover, another application of Proposition~\ref{prop:stochastic-monotonicity},
comparing $\eta^*$ with $0$, shows that
\begin{equation}
\label{eq:local-convergence-2}    
 \IP_{\eta^*} \bigg( \bigcap_{k \in A} 
\bigl\{ \eta_k (t) \geq m_k \bigr\} \bigg) \geq u-\eps.\end{equation}
Combining~\eqref{eq:local-convergence-1} and~\eqref{eq:local-convergence-2}, since $\eps$ was arbitrary,
concludes the proof of Theorem~\ref{thm:local-convergence}.
\end{proof}

\begin{proof}[Proof of Corollary~\ref{cor:infinite_mass_rho}]
Recall from~\eqref{eq:conservation} that $X_1(0) -X_1(t)=\sum_{k \in \N} ( \eta_k(t) - \eta_k (0) )$,
where $\sum_{k \in \N} \eta_k (0) < \infty$ since $X(0) \in \bbXB$. 
Under the hypothesis $\sum_{k \in \N} \alpha_k = \infty$, 
 Lemma~\ref{lem:stationary-mean}
 shows that $\nu_\alpha ( \bbDB ) = 0$.
 Then for every $M \in \RP$ (large)
and every $\eps>0$ (small),
one can  find a finite set~$A_{M,\eps} \subset \N$
such that
\[
 {\nu_\rho} \bigg\{ \eta \in \bbD : \sum_{k\in A_{M,\eps}}\eta_k > M \bigg\}  
    > 1 - \frac{\eps}{2}.
\]
Then, Theorem~\ref{thm:local-convergence} implies that, for all large enough~$t$,
\[
\IP \bigg[\sum_{k\in \ZP} \eta_k(t) > M \bigg] \geq
 \IP \bigg[\sum_{k\in A_{M,\eps}}\eta_k(t) > M \bigg]   
    > 1 - \eps,
\]
which shows the claim.
\end{proof}

\section{Asymptotics of the leftmost particle}
\label{sec:leftmost-particle}

\subsection{Strong law of large numbers}
\label{sec:strong-law}

In the section we consider the asymptotics of $X_1 (t)$, the location of the leftmost particle.
We start with a result that contains the strong law given in Theorem~\ref{thm:strong-law} above.
Recall the   customer random walk $Q$ from Definition~\ref{def:customer-random-walk}.
The main result is that, when~$\cV \neq \emptyset$ and $X(0) \in \bbXB$,
the limit $\lim_{t \to \infty} t^{-1} X_1(t)$ exists and 
is equal to $v_0$, the speed associated with the minimal admissible solution.
Recall that $v_0 \leq 0$; if $v_0 <0$ we say 
the dynamics is \emph{ballistic}.

\begin{theorem}
\label{thm:leftmost-particle-via-customer-walk}
Suppose that~\eqref{ass:positive-rates} and~\eqref{ass:a-inverse-sum} hold, that $\cV \neq \emptyset$, and $X(0) \in \bbXB$. 
\begin{enumerate}[label=(\roman*)]
 \item
 \label{thm:leftmost-particle-via-customer-walk-i}
If $v_0 = 0$, then, for every $k \in \N$,
\begin{equation}
 \label{eq:speed-leftmost-null-recurrent} 
  \lim_{t\to\infty} \frac{X_k(t)}{t} = 0, \as
 \end{equation}
 \item
\label{thm:leftmost-particle-via-customer-walk-ii}
If  $v_0 < 0$ and $\bbar = \infty$, then, for every $k \in \N$,
\begin{equation}
 \label{eq:speed-leftmost-transient} 
  \lim_{t\to\infty} \frac{X_k(t)}{t} = v_0, \as
 \end{equation}
    \item
  \label{thm:leftmost-particle-via-customer-walk-iii}
  Suppose that $Q$ is positive recurrent, and that $\inf_{k \in \N} a_k >0$.
  Then $X_1$ is ergodic, meaning that~\eqref{eq:X1-ergodic} holds.
\end{enumerate}
\end{theorem}
\begin{remarks}
The assumptions that $\cV \neq \emptyset$ and $X(0) \in \bbXB$ cannot be removed, since they ensure the coherence of the cloud of particles and that it is the speed of the minimal admissible solution that governs the asymptotics. Indeed,
it is easy to find examples where~$\cV = \emptyset$ and the leftmost particle separates from the bulk of the cloud, when it will travel at its
intrinsic speed $a_1-b_1$, regardless of the value of~$v_0$. 

As explained before the statement of Theorem~\ref{thm:strong-law}, under assumption~\eqref{ass:a-inverse-sum}, there is
only one accessible admissible solution
started from 
$X(0) \in \bbXB$. If $X(0) \in \bbX \setminus \bbXB$, however, there may be many 
admissible solutions (with $v>0$) that are in principle accessible; cf.~Proposition~\ref{prop:invariant-measure}. 
Our conjecture is that, under mild conditions, if $Q$ is transient or null recurrent, one still observes the speed~$v_0$
for the leftmost particle, while if $Q$ is positive recurrent, which speed $v$ is ``selected'' by the dynamics will depend on the density of the initial configuration (see Section~\ref{sec:literature} for relevant literature). 
Since we are primarily interested in initial conditions in $\bbXB$, we do not examine these cases in the present paper.
\end{remarks}
 
\begin{proof}[Proof of Theorem~\ref{thm:leftmost-particle-via-customer-walk}]
We again use the truncation constructions described in Proposition~\ref{prop:finite-coupling}.
Consider the finite particle system $\xi^{N,0}(t)$ 
  with parameters $a_1, b_1$, $\ldots, a_N, b_N, a_{N+1}$
but with~$b_{N+1}$ set to~$0$. Proposition~\ref{prop:finite-coupling}\ref{prop:finite-coupling-i}  furnishes 
a coupling of $\xi^{N,0}(t)$
and $\xi$ such that,
for all  $k \in \{1,2,\ldots, N\}$ and all $t \in \RP$, $X_k^{N,0} (t) \leq X_k (t)$.
As in the proof of Theorem~\ref{thm:local-convergence}, an application
of Proposition~\ref{prop:finite-cloud} to the system $\xi^{N,0}(t)$
shows that $\lim_{t\to \infty} t^{-1} X^{N,0}_k (t) = v_{N,0}$, a.s.,
for every~$k$, where $v_{N,0} = -\alpha_{N+1}/\beta_{N+1}$.
Then, since $\lim_{N \to \infty}  v_{N,0} = v_0$,
\[ \liminf_{t \to \infty} \frac{X_k(t)}{t} \geq \lim_{N \to \infty}  v_{N,0} = v_0, \as \]
If $v_0 =0$, then this completes the proof of part~\ref{thm:leftmost-particle-via-customer-walk-i},
since for $X(0) \in \bbXB$ we have $\limsup_{t \to \infty} t^{-1} X_k (t) \leq 0$, a.s.

Suppose, on the other hand, that $v_0 < 0$. Consider the finite particle system $\xi^{N,1}(t)$ 
  with parameters $a_1, b_1, \ldots, a_{N+1} , b_{N+1}$. 
 Proposition~\ref{prop:finite-coupling}\ref{prop:finite-coupling-ii}  furnishes  a coupling of $\xi^{N,1}(t)$
and $\xi$  such that,
for all  $k \in \{1,2,\ldots, N\}$ and all $t \in \RP$,
$X_k^{N,1} (t) \geq X_k (t)$.
Another application of 
Proposition~\ref{prop:finite-cloud}
shows that
\[ \lim_{t \to \infty} \frac{X^{N,1}_k (t)}{t} = \frac{1-\alpha_{N+1}}{\beta_{N+1}} = \frac{1}{\beta_{N+1}} + v_{N,0}, \as \]
Under the assumption that $\bbar = \infty$, we can take $N \to \infty$ (along a suitable subsequence)
to conclude that
\[ \limsup_{t \to \infty} \frac{X_k (t)}{t} \leq v_0, \]
which completes the proof of part~\ref{thm:leftmost-particle-via-customer-walk-ii}.

Finally, for part~\ref{thm:leftmost-particle-via-customer-walk-iii},
recall from
Lemma~\ref{lem:customer-walk} that $Q$ is positive recurrent
if and only if $\sum_{k \in \N} a_{k+1} \alpha_k < \infty$.
Consequently, under the condition~$\inf_{k \in \N} a_k > 0$, 
if $Q$ is positive recurrent then $\sum_{k \in \N} \alpha_k < \infty$.
If $\cV \neq \emptyset$, then $v_0 \in \cV$,
but $v_0 = 0$ by~\eqref{eq:v0-recurrnce}; hence $0 \in \cV$, meaning that $\alpha$ is admissible. We may thus apply Corollary~\ref{cor:finitely-supported-convergence} to obtain the result.
\end{proof}

\subsection{Dog and sheep example}
\label{sec:dog-sheep}

The aim of this section is to prove Theorem~\ref{thm:dog_repelled}
on the asymptotics of the ``dog'' particle in the ``dog and sheep'' 
process described in Example~\ref{ex:lattice-atlas-first}. Recall that 
there is a dog (a particle with intrinsic drift
to the right) and,
placed to its right, countably many sheep (particles with zero intrinsic drift), in such a way that the total number of unoccupied sites to the right of the dog is finite, i.e., $X(0) \in \bbXB$.
For Theorem~\ref{thm:dog_repelled}, it is convenient to assume 
that the dog jumps to the left with rate $a<1$
and \emph{attempts} to jump to the right with rate~$1$,
and that the sheep attempt to jump either direction at rate~$1$, i.e.,
$b  = c =1$ in the notation at~\eqref{eq:dog-sheep-rates};
this restriction should not be essential.

\begin{proof}[Proof of Theorem~\ref{thm:dog_repelled}]
For the sake of clarity, we give the proof for initial configurations $X (0) = (X_1 (0), 0)$, $X_1 (0) \in \Z$, that are translates of the Heaviside configuration~$\xheaviside$,
so that~$\eta(0)$ is the all-zero
initial configuration; the general case is completely analogous.

Let us prove the first inequality in~\eqref{eq_dog_repelled}.
If we start from the all-zero
initial configuration, it follows from~\eqref{eq:X1-via-counting} that
\begin{equation}
\label{represent_X_1}
 -X_1(t) = \eta_1(t)+\eta_2(t)+\cdots,
\end{equation}
that is, in the terms of Jackson networks, the (negative) displacement
of the dog equals the total number of customers in the system.
Now, recall the customer random walk introduced in Section~\ref{sec:customer-walk}:
with the assumption~\eqref{eq:dog-sheep-rates} and $b=c=1$, the transition probabilities given at~\eqref{eq:P-def} are $P_{i,i-1} = P_{i,i+1} = 1/2$ for every $i \in \N$, so the customer random walk~$Q$ is simple symmetric random walk on $\ZP$ with absorption at~$0$. Hence, by Proposition~\ref{prop:customer-random-walk}, 
each customer performs a simple random walk (until leaving 
the system, i.e., when absorbed at~$0$),
time-changed by the waiting times at each queue, and the corresponding skeleton
walks are independent. 

With high probability,  by time~$t$ at least order~$t$
 customers will enter 
the system and make it to the second queue.
More specifically, denote that number by $N_t$; 
then, for small enough~$u>0$, a Poisson large-deviation
estimate gives us $\IP ( N_t\geq ut )\geq 1-\re^{-ct}$ for some
positive~$c$.
Then (recall Proposition~\ref{prop:customer-random-walk}\ref{prop:customer-random-walk-ii}), consider
the corresponding \emph{independent} random walks
with rate-2 exponential holding times, and notice that if
by some time~$t_1$ that random walk did not reach the first
queue, then that customer also did not do that (this is because
for actual customers the waiting times \emph{dominate} 
rate-2 exponential random variables). 
Now, a standard argument shows that,
for each such walk, the probability that it does not 
reach the first queue (starting from the second one)
by time~$t+1$ is at least~$\frac{v}{\sqrt{t}}$ for small enough~$v$
(indeed, it reaches $\sqrt{t}$ before reaching~$1$ with 
probability of order $\frac{1}{\sqrt{t}}$, and, starting 
at~$\sqrt{t}$, with a constant probability it does not reach~$1$
in time~$t+1$). Then (given that $N_t\geq ut$) an 
application of the Chernoff bound for the binomial distribution
shows that, by time~$t+1$ with probability at least 
$1-\re^{-c'\sqrt{t}}$, there will be at least
$\frac{uv}{2}\sqrt{t}$ such walks which never reached~$1$.
Therefore, at all moments in the time 
interval~$[t,t+1]$ with probability at least $1-\re^{-ct}-\re^{-c'\sqrt{t}}$,
 there will be at least
$\frac{uv}{2}\sqrt{t}$ customers in the system.
By~\eqref{represent_X_1} and Borel--Cantelli, this implies
the first inequality in~\eqref{eq_dog_repelled}.

To prove the second inequality in~\eqref{eq_dog_repelled},
one can reason 
using a \emph{conditional coupling} such as in~\cite[\S4.2]{brass}, as follows.
 Proposition~\ref{prop:stochastic-monotonicity} implies that
the configuration $\eta(t)$ is stochastically dominated
by the stationary configuration, which is 
described by an i.i.d.\ sequence of geometric random variables 
with expectation~$\frac{1-a}{a}$ (see Example~\ref{ex:lattice-atlas-first}).
It is then straightforward to obtain 
for any fixed $\gamma>\frac{1-a}{a}$ that, 
for any~$t$ (note that the number of extra customers
that can come to $\{1,\ldots,k\}$ during the time interval~$[t,t+1]$
is dominated by a Poisson random variable
with a constant rate),
\begin{equation}
\label{eq:LD_customers} 
\IP\big( \eta_1(s)+\cdots+\eta_k(s)\leq \gamma k
 \text{ for all }s\in[t,t+1]\big) \geq 1 - \re^{-ck},
\end{equation}
with some $c=c(a,\gamma)>0$.

Now, 
fix some $\gamma >\frac{1-a}{a}$;
Borel--Cantelli and~\eqref{eq:LD_customers}  then implies that, a.s., there is a (random) $t_0 \in \RP$ such that, for all $t \geq t_0$,
\[
 \eta_1(t)+\cdots +\eta_{\lfloor t^{(1+\delta)/2} \rfloor}(t)
 \leq \gamma t^{(1+\delta)/2} 
 < t^{\frac{1}{2}+\delta}.
\]
But then, if (recall~\eqref{represent_X_1})
$-X_1(t)\geq t^{\frac{1}{2}+\delta}$ for some $t \geq t_0$, we would necessarily
have that there exists $k>t^{(1+\delta)/2}$
such that $\eta_k(t)>0$. This would mean that, in the 
corresponding Jackson network there is a customer
who ``went too far'' (recall that we assumed that initially
there were no customers at all). 
However, it is straightforward to obtain from
Proposition~\ref{prop:customer-random-walk}\ref{prop:customer-random-walk-ii} together with a large deviation bound
for simple symmetric random walk that, 
with probability at least $1-t\re^{-ct^\delta}$ 
no customer that entered before time~$t$ 
could go further than $t^{\frac{1}{2}+\frac{\delta}{2}}$
by time~$t$.
Again by Borel--Cantelli, we obtain that, for all large enough~$t$,
$\eta_k(t)=0$ for all $k\geq t^{\frac{1}{2}+\frac{\delta}{2}}$. 
By the above observation,
this concludes the proof.
\end{proof}
 \begin{remark}
 \label{rem:dog_repelled_proof} 
    The same proof would work for  the case when 
 there are several dogs. Indeed, in that case the customer random walk~$Q$
 is not simple symmetric random walk,
 but a perturbation obtained by modifying the transition probabilities
 in a finite neighbourhood of~$0$, and this has no impact on the estimates needed above.
 \end{remark}

\section*{Acknowledgements}
\addcontentsline{toc}{section}{Acknowledgements}

The authors thank an anonymous referee for their comments, and Pablo Ferrari for his help with 
Proposition~\ref{prop:invariant-measure}.
The work of MM and AW was supported by EPSRC grant EP/W00657X/1.
SP was partially supported by
CMUP, member of LASI, which is financed by national funds
through FCT --- Funda\c{c}\~ao
para a Ci\^encia e a Tecnologia, I.P., 
under the project with reference UIDB/00144/2020. Revision of the paper was undertaken  during the programme ``Stochastic systems for anomalous diffusion'' (July--December 2024) hosted by the  Isaac Newton Institute, under EPSRC grant EP/Z000580/1.

\printbibliography

\end{document}